\documentclass[12pt]{article}

\usepackage[colorlinks=true,linkcolor=blue,bookmarks=true]{hyperref}
\usepackage{fancyhdr}
\usepackage{makeidx}
\usepackage{amsmath,amssymb,dsfont}
\usepackage{graphicx,color}
\usepackage{enumitem, xspace, float}
\usepackage[usenames,dvipsnames]{xcolor}
\usepackage{graphicx}
\usepackage{dsfont}
\usepackage{cases}
\usepackage{xspace}

\makeindex
\def\abstract#1{{\noindent\textbf{Abstract.} \em #1}
\par
\medskip

}

\definecolor{Labo}{rgb}{0.9,0.9,1.0}
\definecolor{Black}{rgb}{0.00,0.00,0.00}
\definecolor{Red}{rgb}{1.00,0.00,0.00}
\definecolor{Blue}{rgb}{0.00,0.00,1.00}
\definecolor{White}{rgb}{1.00,1.00,1.00}
\definecolor{aqua}{rgb}{0.00,255,255}
\definecolor{Yellow}{rgb}{1.00,0.00,0.00}

\newtheorem{theorem}{Theorem}[section]
\newtheorem{lemma}[theorem]{Lemma}
\newtheorem{corollary}[theorem]{Corollary}
\newtheorem{proposition}[theorem]{Proposition}

\newtheorem{remark}[theorem]{Remark}

\numberwithin{equation}{section}

\newenvironment{proof}[1][Proof]{\medskip\noindent\emph{#1 ---\;}}{\\ 
    \null\  \hfill\textbf{Q.E.D.}\medskip}

\def\e{\varepsilon}
\def\eps{\varepsilon}

\newcommand{\lsc}{l.s.c.\xspace}
\newcommand{\usc}{u.s.c.\xspace}

\def\be{\begin{equation}}
\def\ee{\end{equation}}

\def\R{\mathbb R}
\def\N{\mathbb N}

\def\G{\mathbb G}

\def\Omegb{{\overline \Omega}}
\def\domeg{{\partial \Omega}}
\renewcommand{\H}{\mathcal{H}}

\newcommand{\trp}[1]{{}^t\hspace*{-0.1em} #1}

\renewcommand{\H}{\mathcal{H}}

\def\ue{u_\eps}

\def\xe{x_\eps}
\def\pe{p_\eps}
\def\te{t_\eps}

\def\xb{{\bar x}}
\def\yb{{\bar y}}

\def\tb{{\bar t}}

\def\sb{{\bar s}}

\def\tb{{\bar t}}

\def\u0{u_0}

\def\HT{\HT}

\def\limssup{\mathop{\rm limsup\hspace*{0.1em}\raisebox{0.25em}{$\scriptstyle \ast$}\,}}
\def\limiinf{\mathop{\rm liminf\hspace*{0.07em}\raisebox{-0.15em}{$\scriptstyle \ast$}\,}}

\newcommand{\Dxp }{D_{x'}}
\newcommand{\Ddxp }{D^2_{x'x'}}

\def\HT{{H}_T}

\newcommand\PP{{\mathcal{P}}}

\newcommand{\hyp}[1]{$(\mathbf{H}_\textsc{#1})$}

\newcommand{\LCR}{{\bf (LCR)}\xspace}
\newcommand{\SCR}{{\bf (SCR)}\xspace}
\newcommand{\GCR}{{\bf (GCR)}\xspace}

 \def\sym{{\cal S}^N}

 \def\1{{1\hspace{-1.2mm}{\rm I}}}

  \newcommand{\commentout}[1]{}

\newcommand{\resp}[1]{$[\,resp.$ #1$\,]$\xspace}

\def\Tf{T_f} 

\def\PC1{\mathrm{PC}^{1}(\R^N\times [0,\Tf])}

\newcommand{\wlg}{w.l.o.g.~} \newcommand{\ie}{i.e.~}
\newcommand{\wrt}{w.r.t.~} \newcommand{\cf}{cf.~}

\usepackage{adjustbox}

\newcommand{\smsp}{\ \\[3pt]}

\newenvironment{app}[3]{ \medskip \noindent #1\ --- \emph{#2}~\emph{#3}\nopagebreak}{\medskip}

\newcommand{\keypoint}[2]{
\end{minipage} \end{center} }

\newcommand{\Id}{\mathop{\mathrm{Id}}}

\newcommand{\balpha}{{\bar \alpha}}

\newcommand{\Cea}{C_{\e,\alpha}}
\newcommand{\Ceb}{C_{\e,\beta}}
\newcommand{\Qbl}{{\bar Q^\ell}}
\newcommand{\pQbl}{\partial\Qbl}

\begin{document}

\title{Some Comparison Results for First-Order Hamilton-Jacobi Equations and Second-Order Fully
Nonlinear Parabolic Equations with Ventcell Boundary Conditions}

\author{Guy Barles\thanks{Institut Denis Poisson (UMR CNRS 7013).
Universit\'e de Tours, Universit\'e d'Orl\'eans, CNRS.
Parc de Grandmont.
37200 Tours, France}\ 
\footnote{\texttt{<guy.barles@idpoisson.fr>}}\ \ \ \&\ 
Emmanuel Chasseigne$^*$\footnote{\texttt{<emmanuel.chasseigne@idpoisson.fr>}}}

\date{}
\maketitle

\begin{abstract}
{\footnotesize 
    In this article, we consider fully nonlinear, possibly degenerate, parabolic equations associated
    with Ventcell boundary conditions in bounded or unbounded, smooth domains. We first analyze the
    exact form of such boundary conditions in general domains in order that the notion of viscosity
    solutions makes sense. Then we prove general comparison results, both for first- and second-order equations,
    under rather natural assumptions on the nonlinearities: $(i)$ in the second-order case, the only restrictive assumption
    is that the equation has to be strictly elliptic in the normal direction, in a neighborhood of
    the boundary; $(ii)$ in the first-order one, quasiconvexity assumptions have to be imposed both
    on the equation and the boundary condition, the equation being coercive in the normal direction.
    Our method is inspired by the ``twin blow-up method'' of Forcadel-Imbert-Monneau, that we adapt
    to a scaling consistent with the Ventcell boundary condition.} 
\end{abstract}


\section{Introduction}

Introduced in 1981 by Crandall and Lions \cite{CL} (see also Crandall, Evans and Lions \cite{CEL})
for first-order Hamilton-Jacobi Equations, the notion of viscosity solutions is known to be the
right notion of weak solution to deal with second-order, fully nonlinear, possibly degenerate
elliptic or parabolic equations. Nowadays, the basic theory can be considered as being rather
complete with very general stability results, and in particular the ``Half-Relaxed Limits Method''
which can be powerfully used if the limit equation satisfies a \textit{strong comparison result},
\SCR for short, \ie a comparison result between semicontinuous sub and supersolutions. 

Such \SCR not only provide the uniqueness of solutions, they are also a key tool for obtaining their
existence via the Perron's method of Ishii \cite{Is-Per}, and they exist in almost all the
frameworks: whether the equations are set in the whole space or in bounded or unbounded domains,
with the most classical boundary conditions (Dirichlet, State-Constraint, nonlinear Neumann boundary
conditions, etc.) or for equations involving nonlocal terms (\cite{BI} and references
therein), or equations set in a network or with discontinuities (see \cite{BCbook}  and references
therein). The reader may have a first idea of this theory by looking at the ``User's guide''
of Crandall, Ishii and Lions \cite{Users}; we give more references of \SCR later in this
introduction. 

Roughly speaking, a \SCR is the analog of the Maximum Principle for classical (in other words, smooth) solutions
and, with few additional technical assumptions, a \SCR exists for any classical situation where the
equation, together with the associated boundary condition, formally satisfy the Maximum Principle. 
Of course, in the framework of viscosity solutions---which we use here---the boundary condition has
to be understood a priori in the relaxed sense given by viscosity solutions theory: either the inside equation
or the boundary condition should hold for both the subsolution and the supersolution,
see~\cite{Users}. This particularity, which is, in general, a difficulty for proving such \SCR,
is now well-addressed in most classical situations. 
However, coming back to Maximum Principles, in the case of \textit{Ventcell boundary conditions}, no
\SCR was available in the literature so far. We explain why such boundary conditions create a
specific difficulty later on. 

\bigskip

\noindent\textit{The aim of this article ---} We provide here the very first \SCR for Ventcell
boundary conditions in the viscosity solutions' framework. We immediately point out that we are
able to do so both for first-order and second-order equations, under some reasonable assumptions on the equation and
the boundary condition.
We also recall that a \SCR is actually a ``global'' comparison result, providing
comparison in all the domain---and we use below the notation \GCR instead of
\SCR to emphasize this global comparison. However, as in \cite{BCbook}, we reduce it to a ``local
comparison result'', \ie to a comparison result which holds in a small neighborhood of each point; we denote below such
local result by \LCR. This reduction to \LCR allows us to mainly consider the case of equations set in an
half-space and Section~\ref{statements} shows that our results easily extend to the case of general
regular domains via localization arguments and a straightforward local flattening of the boundary. 

\bigskip

\noindent\textit{The Ventcell boundary condition ---}
Now, in order to be more specific, we consider general fully nonlinear, possibly degenerate,
parabolic equation of the form 
\begin{equation}\label{Eqn}
u_t + F(x,t,D_x u, D^2_{xx} u)= 0 \quad \hbox{in }\Omega \times (0,T) ,
\end{equation}
where $\Omega$ is a bounded or unbounded domain of $\R^N$, the solution $u$ is a real-valued
function defined on $\Omegb \times [0,T)$, $u_t$, $D_x u, D^2_{xx} u$ denote its first and
second-derivatives with respect to $t$ and $x$ respectively. Finally, $F: \Omegb\times [0,T) \times
\R^N \times \sym \to \R$, where $\sym$ is the space of $N\times N$-symmetric matrices, is a
real-valued, continuous function satisfying the ellipticity assumption
\begin{equation}\label{ellip}
F(x,t,p_x,M_1)\leq F(x,t,p_x,M_2) \quad \hbox{if  }M_1\geq M_2 ,
\end{equation}
for any $x\in \Omegb$, $t\in [0,T)$, $p_x\in \R ^N$, $M_1,M_2\in \sym$, where ``$\geq$'' denotes the
partial ordering on symmetric matrices.

In order to introduce the Ventcell boundary condition, we first consider the case when
$\Omega$ is a half-space of $\R^N$ and, to fix ideas, we choose
\begin{equation}\label{Om-HS}
\Omega:=\{x=(x',x_N) \in \R^{N-1}\times \R, x_N>0\}.
\end{equation}
In this context, the Ventcell boundary condition for Equation~\eqref{Eqn} has the form
\begin{equation}\label{Vbc-HS}
-u_{x_N} + G(x',t,\Dxp  u, \Ddxp   u)= 0 \quad \hbox{on }\domeg \times (0,T) ,
\end{equation}
where $G$ satisfies similar assumptions as $F$, in particular an ellipticity property like
\eqref{ellip}. We point out that, in our context, $-u_{x_N}$ is nothing but the
normal derivative of $u$ on $\domeg \times (0,T)$ and therefore \eqref{Vbc-HS} is some kind of
Neumann type boundary condition.
However, this comes with an unusual dependence in the second-order
tangential derivative $\Ddxp   u$. This particularity is, of course, the main originality and
difficulty of Ventcell boundary conditions. 

\bigskip

\noindent\textit{The case of a general domain ---}
If $\Omega$ is a general smooth domain\footnote{We will precise which type of regularity we
impose later on.}, the exact form of such boundary condition and the assumptions they have to satisfy are
less clear, for at least two reasons.

First, at a point $x$ of the manifold $\domeg$, the condition has to depend on the Hessian matrix---relatively
to $\domeg$---of the solution $u:\Omegb \to \R$ but it is well-known that the definition of such
Hessian matrix on a manifold is not completely straightforward: not only does it depend on $D^2_T u$,
the $N\times N$-symmetric matrix corresponding to the restriction of the quadratic form $h\mapsto
D^2u(x)h\cdot h$\ \footnote{Here and throughout this article, $v_1\cdot v_2$ stands for the standard
euclidian scalar product of $v_1,v_2\in \R^N$.} to $T_x\domeg$ (the tangent space of $\domeg$ at
$x$), but it also depends on the curvatures of $\domeg$ at $x$.

For the time being, we just write the boundary condition in the general form
\begin{equation}\label{Vbc-GC}
\G(x,t,D u, D^2_T u)= 0 \quad \hbox{on }\domeg \times (0,T) ,
\end{equation}
where we recall that, if $n(x)$ denotes the outward normal to $\domeg$ at $x$ and $\Id$ is the
$N\times N$ Identity matrix, $D^2_Tu$ is obtained by using the projection onto $T_x\domeg$, whose
matrix is given by $\Id-n(x)\otimes n(x)$; hence the formula 
\[ D^2_T u(x):=\big(\Id-n(x)\otimes n(x)\big)\, D^2u(x)\, \big(\Id-n(x)\otimes n(x)\big)\,.\] 
We refer the reader to Section~\ref{GVbc-GD} where we explain in an elementary way what kind of
assumptions a general boundary condition like \eqref{Vbc-GC} should satisfy in order to be a ``good"
Ventcell boundary condition.

Of course, these restrictions
are of two types: the first ones are just basic compatibility conditions in order that
\eqref{Vbc-GC} is actually consistent with the Maximum Principle, and therefore that the notion of
viscosity solutions makes sense. The second ones are related to comparison results and the main
assumption consists in imposing that \eqref{Vbc-GC}  can be reduced to \eqref{Vbc-HS} by $(i)$ a
suitable change of coordinates which flattens the boundary and $(ii)$ a suitable monotonicity
property in $u_{x_N}$ after the change of coordinates in order to be able to write down the boundary
condition as \eqref{Vbc-HS}. In that way, as we explain it in Section~\ref{br}, the main step in a
comparison proof in a general domain is nothing but a local comparison result for \eqref{Vbc-HS}.

More generally, we want to point out a key idea in this article: all the local properties for
\eqref{Eqn}-\eqref{Vbc-GC} are obtained from \eqref{Eqn}-\eqref{Vbc-HS} since the mecanism
$(i)$-$(ii)$ we described above allows to reduce to this case.
Now, concerning global properties such as the existence of sub and supersolutions, which are needed
either for localizing the comparison proof or for Perron's method, we use only basic assumptions on
$\G$.  In fact, as this description suggests, most of the results are proved for
\eqref{Eqn}-\eqref{Vbc-HS}.
 
\bigskip

{
\noindent\textit{The literature on Ventcell boundary conditions ---}
Ventcell (or Ventcel) boundary conditions appear in the mathematical literature in different contexts. 
First, in modelling, these boundary conditions often arise in the study of asymptotics for thin
layers on the boundary; the results in this direction are either numerical (see, for example, Dambrine and Pierre~\cite{DP} and 
references therein) or more theoretical, using typically the Lax-Milgram Theorem in the elliptic case (Bonnaillie-No\"el et al.~\cite{{BDHV}} and references therein).
We point out that, in this direction, most of the references are concerned with numerical issues.

Closer to our motivations, these boundary conditions are shown to be naturally associated to {\em Waldenfels operators}, i.e.
to (local or nonlocal) operators which satisfy the Maximum Principle: we refer the reader to Taira~\cite{Tk} or to Priouret~\cite{Pr}
and references therein. These types of works use either classical analysis methods (Sobolev or Besov spaces, semi-groups 
theory, etc.) or connections with probability (Markov or diffusion processes) as in \cite{Pr}. The thesis of El Karoui~\cite{ElK}
seems closer to our purpose by showing that such boundary conditions are associated with diffusion processes with a reflection on $\domeg$ (see also Petit~\cite{Pf}).}

\bigskip

\noindent\textit{The difficulty to handle such boundary conditions ---}
Maybe the easiest way to explain why getting a comparison result for \eqref{Eqn}-\eqref{Vbc-HS} in
the viscosity solutions framework is difficult is to recall the method which is used to treat
nonlinear Neumann boundary conditions, \ie the case when $G$ does not depend on $\Ddxp   u$.
Initiated by Lions \cite{Li-Neu} for standard linear Neumann and oblique derivatives boundary
conditions, the method was then generalized under slightly different forms in the nonlinear setting
(with slightly different assumptions) by Ishii \cite{Is-Neu} and Barles \cite{Ba-Neu}. 

Of course, the difficulty comes from the condition at the boundary and the comparison proof consists
in building a test-function for which the Neumann boundary condition cannot hold.  With such a
property, the $F$-inequalities necessarily hold true, both for the sub and the supersolution
and, if the test-function satisfies suitable estimates, the conclusion follows. 

In order to follow this strategy, a key point is that the (weak) derivatives of the sub and
supersolution are nothing but derivatives of the test-function at the maximum or minimum point.
Therefore, these derivatives can be directly read on the test-function and inserted into the equation. 
However, for second-order terms, any comparison proof for viscosity solutions uses the
Jensen-Ishii Lemma (\cite{J2nd,I2nd}) which provides the second derivatives for the sub and
supersolution in a somewhat abstract way. In particular, there is no way to build a
test-function for which the boundary condition cannot hold.   

\bigskip

\noindent\textit{How to turn around the difficulty? The different strategies ---} In order to obtain
a comparison result for Ventcell type problems, despite of the key difficulty which is described
above, one may imagine two main strategies. 

\noindent \textbf{1.} The first one, which may appear at first glance as the most natural and
simplest one in the half-space case, consists in using a ``tangential regularization'' in the
$x'$-variable, at least for the subsolution, and some convexity assumptions on the
nonlinearities. The idea is to take advantage of having a smooth subsolution to give a sense to the
second-derivative $\Ddxp   u$ and to use it in the comparison proof. Such regularization is
already used for first-order equations in \cite{BCbook} and, to perform such regularization, the
flat boundary case is the most natural framework to begin with.

Unfortunately, this approach comes with two main defects: on one hand, even usual basic
regularization by sup and/or inf-convolution may require restrictive assumptions. This is true in particular for
second-order equations, or if we want to perform at the same time a sup-convolution on the
subsolution and an inf-convolution on the supersolution. And as a consequence, we end up with
unreasonnable hypotheses for the extensions to general domains.

The second defect is related to a further smoothing of subsolutions---in a convex framework---by
usual convolution with a mollifying kernel: even if it seems {\em formally} obvious that such
regularization procedure can be applied, the boundary condition taken in the viscosity sense and the
fact that we have to use a doubling of variables to rigourously establish the subsolution property
for the regularized function create difficulties which we were not able to handle, except in the
(very) particular case of the Appendix~(Section~\ref{app:regsubsol}).

\medskip 

\noindent \textbf{2.} The second strategy consists in trying the ``twin blow-up'' introduced
recently by Forcadel, Imbert and Monneau \cite{FIM} with the hope that the blow-up provides simpler
equations and that we could conclude through a careful examination of the super- and sub
differential (or of some suitable super and subjets) of the sub and supersolution respectively,
at a maximum point of their difference, following new arguments introduced recently by Lions and
Souganidis \cite{LiSo1,LiSo2}.

But here also we face a problem since the Lions-Souganidis arguments turn out to be mainly
one-dimensional while here, because of the Ventcell boundary condition, we have to take into account
(at least) a first derivative in $x_N$ and a second derivative in $x'$. To do so, we would need a
multi-dimensional Lions-Souganidis argument but we were unable to obtain it.

\medskip

What can be done when both strategies fail? In some sense, we combine them here, with some ad hoc adaptations
due to the Ventcell framework. The core of our comparison proofs---either in the cases of a
first-order equation or a second-order one---is an extension of the strategy of Forcadel, Imbert and
Monneau with a scaling which is adapted to the Ventcell boundary condition. More precisely, our
scheme of proof in the case of \eqref{Eqn}-\eqref{Vbc-HS} is the following:
\begin{enumerate}
    \item[$(i)$] We use an almost classical doubling of variables method but, here, in an unusual
        way: it is not the main step anymore, but some kind of ``preparation'' to the ``twin
        blow-up'' argument. Indeed, the doubling of variables allows us to reduce to the case when
        the maximum points are both on the boundary---hence preparing the twin blow-up. But it also
    gives additionnaly some useful estimates to perform the blow-up.  \item[$(ii)$] The twin blow-up
        is done in a different way here since it has to be adapted to the Ventcell boundary
        condition: we use different scalings in the tangential directions $(x',t)$ and in the normal
    one, \ie for $x_N$. We perform it not only in the equation and boundary conditions, but also in
the maximum point property related to the doubling of variables, providing useful estimates.
\item[$(iii)$] In the present situation, passing to the limit in the blow-up procedure does not
    allow to reduce to a one-dimensional problem, again because of the Ventcell boundary condition
        which mixes tangential and normal variables. Here, getting the conclusion follows different
        paths according to the cases when the equation inside the domain is a first-order or
        second-order one. In the former case, we use the first strategy described above by using a
        regularization of the subsolution. In the latter case, a suitable adaptation of the
        Jensen-Ishii Lemma allows us to take advantage of the particularity of the limiting problem.
\end{enumerate}

In order to be able to apply this strategy, we use two specific assumptions in addition to the
classical hypotheses which classically appear in such comparison results: either the equation is a
first-order equation and we require a normal coercivity property (\cf \hyp{NC} in
Section~\ref{sec:hyp}) with suitable quasiconvexity properties for $F$ and $G$, or it is a
second-order equation and we require a strong ellipticity in the normal direction, \cf \hyp{NSE} in
Section~\ref{sec:hyp}. 

A typical example that fits into the framework of this paper is the following one, posed in
$\Omegb\times[0,T]$ where $\Omega:=\{(x,y)\in\R^2:x>0,\ y\in\R\}$ \[ \begin{cases} u_t-\mathop{\rm
    Tr}\Big(A(x,y)D^2u)\Big)+b(x,y)|Du| = f(x,y) & \text{ in }\Omega\times(0,T]\,,\\[2mm] \hfill
    -\dfrac{\partial u}{\partial x}-\dfrac{\partial^2 u}{\partial y^2} = g(y)  & \text{ on
}\partial\Omega\times(0,T]\,,\\[2mm] \hfill u((x,y),0)=u_0(x,y) & \text{ in }\Omega\,.  \end{cases}
\] where we assume that $A=\sigma\trp{\sigma}$\ \footnote{Here and below $\trp{\sigma}$ denotes the
transposed matrix of the matrix $\sigma$.}, where $\sigma, b$ are bounded, Lipschitz continuous
functions on $\overline{\Omega}$ and $f,g$ are bounded and continuous on $\overline{\Omega}$ and
$\partial \Omega$ respectively. In order to satisfy our additional assumptions, we need that, either
$A\equiv 0$ and $b(x,y)\geq\alpha>0$ on $\overline{\Omega}$, or $A(x,y)$ is a symmetric positive
matrix and, with $e_N=(1,0)$, $A(x,y)e_N\cdot e_N \geq\alpha>0$ on $\overline{\Omega}$. 

It is not clear to what extent these additional assumptions, namely \hyp{NC}-\hyp{QC-G} and
\hyp{NSE}, are necessary. However, $(i)$ they really play a key role in our proofs of the comparison
results both in the first- and second-order case; $(ii)$ N. El Karoui \cite{ElK} used the
probabilistic analogue of \hyp{NSE} in her work; $(iii)$ Proposition~\ref{VBC-uec} in
Section~\ref{lpvc} shows that, if \hyp{NSE} holds then the Ventcell boundary condition is satisfied
in a strong sense.  In any case, one may think that \hyp{NC} or \hyp{NSE} ensures that the Ventcell
boundary condition is seen in a right way.

We conclude this introduction by a remark: the approach that we use here allows to treat, as a
special case, Neumann boundary condition---typically $-u_{x_N}+G(x,t,Du)=0$. However, some of the
assumptions we use in order to obtain comparison results---see \hyp{NC} and \hyp{NSE} in
Section~\ref{sec:hyp}---are clearly too restrictive compared to the ones which are used in the
literature on the Neumann case. But maybe some specific modification of our arguments allows not
only to recover all the known results but also to improve them.

\bigskip

\noindent\textit{Organization ---} In Section~\ref{GVbc-GD}, we define what a ``good'' Ventcell
boundary condition is in a general, non-flat domain. Section~\ref{br} is devoted to present basic
assumptions, notations and results to prepare the three next sections which are devoted to first
state and then prove the comparison results. In particular, we recall how to reduce the global \SCR
to a local one.  The statements of these results are provided in Section~\ref{statements} and then
we prove them in the case of first-order equations in Section~\ref{case1} and in the case of
second-order equations in Section~\ref{case2}, the proofs in these two cases being rather different
even if they use similar common ingredients. Finally, in Section~\ref{FROR}, we provide further
results, we mention some open questions and we sketch simpler proofs under more restrictive
assumptions.

\noindent{\bf Acknowledgements.} {\em This research was partially funded by l’Agence Nationale de la
Recherche (ANR), project ANR-22-CE40-0010 COSS. 

The authors would like to warmly thank the reviewer for all the work he/she did to help us improve our
article: he clearly read it very carefully and made numerous relevant comments. We owe the referee
most of the valuable improvements that were made in this article.}

\section{The Ventcell Boundary Condition in General Domains}\label{GVbc-GD}

As we already mentioned in the introduction, contrarily to the case of classical (Dirichlet,
Neumann, etc.) boundary conditions, the Ventcell case is particular because of the dependence in the
Hessian matrix of the solution \textit{on the boundary}. For a general boundary condition like
\eqref{Vbc-GC}, we have to investigate under which types of assumptions this boundary condition is
consistent with the Maximum Principle, and therefore for which a notion of viscosity solutions makes
sense. And to do so, we would have to use the definition of an Hessian matrix on a codimension~$1$
manifold---which is not completely straightforward.

Instead of doing that, in this section, we have chosen to present in the simplest possible way the
conditions on the function $\G$ in order that it yields a ``good'' Ventcell boundary condition. Then
we show how \eqref{Vbc-GC} can be locally reduced to \eqref{Vbc-HS} by a suitable flattening of the
boundary.

We argue assuming that the boundary $\domeg$ is as smooth as necessary---we refer the reader to
\hyp{$\Omega$} below (see Section~\ref{sec:hyp}) for a more precise assumption concerning the
regularity of the boundary.  We recall that the smoothness of $\domeg$ implies that $d$, the
distance function to $\domeg$, is smooth on $\Omegb$ in a neighborhood of $\domeg$, and that
$Dd(x)=-n(x)$ on $\domeg$; we may keep the notation $n(x)$ for $-Dd(x)$ even if $x$ is not on the
boundary. Moreover, the distance function carries other geometrical information: indeed, for any $x
\in \domeg$, the eigenvalues of $D^2d(x)$ are $-\kappa_1, -\kappa_2,\cdots, -\kappa_{N-1}$, the
principal curvatures of $\domeg$ at $x$ (See Gilbarg and Trudinger \cite{GT}, Section~14.6).

\subsection{Consistency with the Maximum Principle}

In order to answer this first question, we adopt a viscosity solution point of view---or a Maximum
Principle one---and, at least formally, we look at maximum points of $u-\phi$ where $u$ is candidate
to be a subsolution (that we assume to be smooth at first), and $\phi$ is a smooth test-function. 

We drop the $t$-variable since it plays no role in the boundary condition but the reader may easily
check that $t$ can be taken into account as any tangent variable, and so is $u_t$ which is a tangent
derivative on the boundary $\domeg\times (0,T)$.

\begin{proposition}\label{prop:ventcell.1} Let $x\in \domeg$ be a local maximum point on $\Omegb$ of
    $y\mapsto (u-\phi)(y)$. Then the following first and second-order inequalities hold:
    \begin{equation}\label{prop-bord} \begin{aligned} (i) &\quad \frac{\partial u}{\partial
        n}(x)\geq \frac{\partial\phi}{\partial n}(x)\text{ and }
    D u(x)=D\phi(x)+\lambda n(x)\; \text{for some }\lambda\geq 0\,,\\
    (ii) & \quad D^2 u(x)+\frac{\partial u}{\partial n}(x) D^2d(x) \leq  D^2 \phi(x)+\frac{\partial
\phi}{\partial n}(x) D^2d(x)\text{ in }T_x\partial\Omega\,.  \end{aligned} \end{equation}
    \end{proposition}

\begin{proof} If $x\in \domeg$ is a local maximum point on $\Omegb$ of $u-\phi$, let us first notice
    that the first inequality in $(i)$---the normal direction one---is classical: \[ \frac{\partial
    (u-\phi)}{\partial n}(x) \geq 0\,.\] For the tangential direction, we consider a smooth path
    $\chi:(-\eta,+\eta)\to \domeg$ such that $\chi(0)=x$. Since $0$ is a maximum point of $s\mapsto
    (u-\phi)(\chi(s))$, by differentiating it follows that $D (u-\phi)(x)\cdot \chi'(0)=0$. 

    Moreover, using that $d(\chi(s))=0$ and differentiating this equality at $s=0$ implies that $D
    d(x)\cdot \chi'(0)=0$; in other words, $\tau=\chi'(0)$ belongs to $T_x\partial\Omega$. Hence, by
    choosing all possible paths $\chi$ as above, we deduce that, for any $\tau\in
    T_x\partial\Omega$, $D (u-\phi)(x)\cdot \tau=0$.  Therefore, there exists some $\lambda\in\R$
    such that $D u(x)=D\phi(x)+\lambda n(x)$ and necessarily $\lambda\geq0$ from the normal
    inequality we recalled above, leading to $(i)$.

    We now turn to the second-order condition. Using that $h(s):=(u-\phi)(\chi(s))$ has a maximum at
    $s=0$, the second-order condition yields \begin{equation}\label{ineq.second.order.chi}
    h''(0)=D^2 (u-\phi)(x) \chi'(0)\cdot \chi'(0)+D (u-\phi)(x)\cdot \chi''(0) \leq 0\,.
    \end{equation} Notice that $D(u-\phi)(x)\cdot\chi''(0)=\lambda n(x)\cdot\chi''(0)=-\lambda
    Dd(x)\cdot\chi''(0)$ and, using the second-order derivative of $d(\chi(s))=0$, we also have \[
        D^2 d (x) \chi'(0)\cdot \chi'(0)+D d (x)\cdot \chi''(0)=0\,.\] Gathering these informations
    and denoting by $\tau$ any vector $\chi'(0)\in T_x\partial\Omega$ as above, we arrive at \[
        h''(0)= D^2 (u-\phi)(x) \tau \cdot \tau+\lambda D^2d(x) \tau \cdot  \tau\leq0\,.\] Finally,
    since $\displaystyle \lambda=\frac{\partial (u-\phi)}{\partial n}(x) $, we arrive at
    \begin{equation}\label{prop2-bord} D^2 u(x)+\frac{\partial u}{\partial n}(x) D^2d(x) \leq D^2
        \phi(x)+\frac{\partial \phi}{\partial n}(x) D^2d(x)\,, \end{equation} on the tangent space,
which is $(ii)$.  \end{proof}


\noindent\textit{Consequences on $\G$ ---} In order to take into account
    Inequalities~\eqref{prop-bord} in a proper way, \ie in order to have \[ \G\big(x,t,D\phi(x)
    ,D^2_T\phi(x)\big)\leq \G\big(x,t,Du(x),D^2_T u(x)\big)\leq 0\,,\] we have to require two
    properties on $\G$: on one hand, it is natural to write $\G$ as
    \begin{equation}\label{eqn:formG} \G(x,t,p,M_T):= \tilde G \big(x,t,p,M_T+p\cdot n(x)
    D^2d(x)\big)\,, \end{equation} for any $x\in \domeg$, $t\in [0,T)$, $p\in \R^N$ and $M_T$, where
    we recall that $M_T$ is defined for $M\in \sym$ by $M_T=(\Id-n(x)\otimes n(x))M(\Id-n(x)\otimes
    n(x))$. We remind the reader that $D^2_T d(x)=D^2d(x)$ since
    $D^2d(x)Dd(x)=-D^2d(x)n(x)=0$, this equality coming from the fact that $\vert Dd(x)\vert^2=1$.
    Of course, we have to assume that the function $\tilde  G$ is elliptic in its last variable; in
    other words, it is non-increasing in this variable in the sense of \eqref{ellip}
    \footnote{This ellipticity requirement is expected since it was expected for $\G$.}.

On the other hand, especially for \eqref{prop-bord}-$(i)$, we have to assume that, for any $\lambda
    \geq 0$, $x\in \domeg$, $t\in [0,T)$, $p\in \R^N$ and $M\in \sym$ \[ \tilde G \big(x,t,p+\lambda
    n(x),M_T+p\cdot n(x) D^2d(x)\big)-\tilde G \big(x,t,p,M_T+(p\cdot n(x)) D^2d(x)\big)\geq 0\,.
    \]

\medskip

Of course, these basic conditions are not even sufficient to define a nonlinear Neumann boundary
    condition---\ie for the case where $\tilde G(x,t,p,M_T)$ does not depend on $M_T$. They have to
    be reinforced in order to get a ``good'' Ventcell boundary condition, in particular we will
    require the more restrictive assumption that { there exists $\bar c>0$ such that, for all
    $x,t,p,M_T,\lambda$ as above\footnote{In particular, $\lambda \geq 0$.},
    \begin{equation}\label{ineq.gvbc} \tilde G \big(x,t,p+\lambda n(x),M_T\big)- \tilde G
    \big(x,t,p,M_T\big)\geq \bar c \lambda.  \end{equation} In other words, under this assumption,
    the boundary condition takes a form similar to \eqref{Vbc-HS}, with a constant $\bar c >0$
    multiplying $u_{x_N}$.} We refer to Section~\ref{sec:hyp} for the exact hypotheses and more
    details.

\subsection{Reduction to a flat comparison result}\label{RFB}

Now we turn to the second question and to do so, we examine some special change of coordinates which
maps $\{y_N=0\}$ in a neighborhood of $0\in \R^N$ into $\domeg$. If $\psi$ is such a diffeomorphism,
we change it into \[ \Psi(y',y_N): = \psi(y',0)+y_N Dd(\psi(y',0)) ,\] in that way, we have
$d(\Psi(y',y_N))=y_N$ (for $|y_N|$ small enough) . Then we set \[ v(y',y_N)=u(\Psi(y',y_N))\,.\] { \textbf{N.B.} In the
following, with a slight abuse of notations we identify the tangential gradients of the form
$(p_T,0)$ with $p_T$, similarly we identify $D_{y'}v(y',0)$ and $(D_{y'}v(y',0),0)$ and finally
$D^2_{y'y'}v((y',0))$, identified with a $N\times N$-matrix with zeros at the last line and column.}
\begin{proposition} The derivatives of $v$ are given by
    \begin{flalign*}
        (i) \quad \dfrac{\partial v}{\partial y_N}=Du(x)\cdot Dd(x)\; , \; 
        D_{y'}v(y',0)= \trp{D\Psi}(y',0) D_Tu(x)\,. && 
    \end{flalign*}
    \begin{flalign*}
            (ii)  \quad  D^2_{y'y'}v(y',0) =   \trp{D\Psi}(y',0)\Big[D^2u(x)+ \frac{\partial u}{\partial n}(x)
        D^2 d (x)\Big]D\Psi(y',0) + a(x)(D_T u(x))\,, 
    \end{flalign*} 
    for some linear map $a(x)$ having the same
    regularity in $x$ as $D^2 \Psi$.  
\end{proposition}

We refer the reader to the second section of the ``User's guide'' of Crandall, Ishii
and Lions \cite{Users} where related results are given (see in particular (2.15) in \cite{Users}).

\begin{proof} Let us compute the $y'$-derivatives of $v$ for $y_N=0$ in a direction $h=(h',0)\in
    \R^{N}$.  Using the notation $x=\Psi(y',y_N)$ to have simpler formulas, we get \[
        \begin{aligned} &  \frac{\partial v}{\partial y_N}(y',0)=Du(x)\cdot Dd(x)\; ,\;
            D_{y'}v(y',0)\cdot h' = Du(x)\cdot D\psi(y',0) h'\,,\\[2mm] &  D^2_{y'y'}v(y',0)h'\cdot
        h' = D^2u(x)D\psi(y',0) h' \cdot D\psi(y',0) h'\\
            & \hspace*{4cm} + Du(x)\cdot D(D\psi(y',0)) (h',h')\,.
        \end{aligned} \] 
    Next, applying these formulas to $y_N=d(\Psi(y',y_N))$, in other words,
    taking $u=d$, we obtain \[ \begin{aligned} 0 &= Dd(x)\cdot D\psi(y',0) h' \,,\\ 0 &= D^2 d
    (x)D\psi(y',0) h'\cdot D\psi(y',0) h' + Dd(x)\cdot D(D\psi(y',0)) (h',h')\,.  \end{aligned} \]
    Coming back to the first-order derivatives of $v(y',0)$, since $h'$ is arbitrary we deduce that
    $D_{y'}v(y',0)=\trp{D\psi}(y',0) Du(x) = \trp{D\Psi}(y',0) D_Tu(x)$ since $\trp{D\psi}(y',0)
    Dd(x)=0$ (we use here the aforementioned abuse of notations). This yields directly $(i)$.

    Now we decompose $Du=\dfrac{\partial u}{\partial n}n(x)+D_Tu$. Using that $n(x)=-Dd(x)$ we see
    that \[ \begin{aligned} \frac{\partial u}{\partial n}(x) n(x)\cdot D(D\psi(y',0) h') &= -
        \frac{\partial u}{\partial n}(x) Dd(x)\cdot D(D\psi(y',0) h') \\ &= + \frac{\partial
        u}{\partial n}(x)   D^2d(x)D\psi(y',0)h'\cdot D\psi(y',0)h'\,.  \end{aligned} \] Gathering
        everything we obtain \begin{align*} D^2_{y'y'}v(y',0)h'\cdot h' = & \left[D^2u(x)+
            \frac{\partial u}{\partial n}(x) D^2 d (x)\right]D\psi(y',0) h' \cdot D\psi(y',0) h' +
            \\  & \ \qquad \qquad \qquad D_T u(x)\cdot D(D\psi(y',0)) (h',h')\,.  \end{align*}
            Finally, since for $y_N=0$, we have $ D\psi(y',0) h'= D\Psi(y',0) h$ and since this
            vector is arbitrary in $T_x\partial\Omega$, we deduce that \[  D^2_{y'y'}v(y',0) =
            \trp{D\Psi}(y',0)\left[D^2u(x)+ \frac{\partial u}{\partial n}(x) D^2 d (x)\right]
D\Psi(y',0) + a(x)(D_T u(x)),\] where $a(x)$ acts linearly on $D_Tu$ and it has the same regularity
in $x$ as $D^2 \Psi$. Hence $(ii)$ holds.  \end{proof}

\noindent\textit{Consequences on $\G$ ---} These properties show that the ``flat'' Hessian matrix
$D^2_{y'y'}v(y',0)$ corresponds to $\displaystyle D_T^2u(x)+\frac{\partial u}{\partial n}(x) D^2 d
(x)$ through the change of coordinates modulo a term depending only on $D_Tu(x)$, the latter
corresponding to $D_{y'}v(y',0)$. Moreover, this formula can easily be inverted.

More precisely, if $u$ is a subsolution \resp{super-solution} of \eqref{Vbc-GC}, then $v$ is a
subsolution \resp{super-solution} of 
\[ \tilde G \Big(\Psi(y),t, P\big(y,t,D_y
v(y,t)\big),\mathcal{M} \big(y,t,{D_{y'}}v(y,t), D^2_{y'y'}v(y,t)\big)\Big)=0\,,\] 
where, for
$y=(y',0) \in \R^N$ close to $0$, $t\in [0,T]$, $q=(q',q_N)$ with $q'\in \R^{N-1}$ and for any
$(N-1)\times (N-1)$-symmetric matrix $M_T$ 
\begin{equation}\label{def:P} 
    P\left(y,t,q\right) = (\,\trp{D\Psi})^{-1}(y)(q',0)- q_Nn(\Psi(y)), 
\end{equation} 
and $\mathcal{M} \big(y,t,q', M_T\big)$ is
given by \[(\,\trp{D\Psi})^{-1}(y)\left[M_T-a(\Psi(y))
(\,\trp{D\Psi})^{-1}(y)(q',0)\right](D\Psi)^{-1}(y)\,.\]

Two remarks on this admittedly complicated formula: on one hand, in order to recover the term
$-\partial_{y_N} v$, one can use \eqref{ineq.gvbc}; this is the purpose of Lemma~\ref{V-flat} below.
On the other hand, the presence of the term $a(\Psi(y))(\,\trp{D\Psi})^{-1}(y)D_{y'}v(y,t)$ perturbs
the assumption we have to impose on $\tilde G$ to be able to use the Jensen-Ishi Lemma and justify
the unusual form of \hyp{Cont} below.

This allows to show that a ``good'' Ventcell boundary condition---in the sense of
Section~\ref{sec:hyp}---is locally equivalent to a ``good'' Ventcell boundary condition in the case
of a flat boundary. Moreover, the result below proves that the boundary condition can be reduced to
the form \eqref{Vbc-HS} with suitable properties on the nonlinearities. 

\begin{lemma}\label{V-flat} Let us assume that $\tilde G$ is a continuous function, which satisfies
    \eqref{ineq.gvbc}, and that $\domeg$ is smooth.  Then there exists a continuous function $G$
    such that \[ \tilde G \big(\Psi(y),t, P(y,t,p'-\lambda e_N),\mathcal{M} (y,t,p', M_T)\big)\] has
    the same sign as \[ -\lambda + G(y,t,p',M_T)\,.\] As a consequence, an equation with the
    boundary conditions $G$ and $\tilde G$ have the same subsolutions and the same supersolutions.
    Moreover, if $\tilde G$ satisfies the hypothesis \hyp{Gen} and/or \hyp{Cont} which are given
    below, then $G$ satisfies them too.  Finally, if $(p,M_T) \mapsto\tilde G(x,t,p,M_T)$ is
quasiconvex for any $x,t$, then so is $G$ \wrt $p',M_T$.  \end{lemma}

\begin{proof} We first notice that we can assume that
    $\bar c=1$ by dividing $\tilde G$ by $\bar c$. Then, if $\mathcal{D}:=\R^{N-1}\times[0,T]\times
    \R^{N-1}\times \mathcal{S}^{N-1}$, we consider the function $f:\mathcal{D}\times \R \to\R$
    defined by \[ f(X,\lambda) := \tilde G \big(\Psi(y),t, P(y,t,p'-\lambda e_N),\mathcal{M}
    (y,t,p', M_T)\big)\] where $X=(y,t,p',M_T)$. 

    The property of $\tilde G$ implies that, for all fixed $X$ and for all $\lambda' \geq \lambda$,
    we have \[f(X,\lambda')-f(X,\lambda)\leq - (\lambda'-\lambda).\] Hence, for all fixed $X$, the
    function $\lambda \mapsto f(X,\lambda)$ is a one-to-one function from $\R$ into $\R$ and there
    exists a unique $G(X)$ such that \[ f(X,G(X))=0,\] and clearly $f(X,\lambda)$ has the same sign
    as $-\lambda+G(X)$.

    For the properties of $G$, we just write that, if $X,X'$ satisfy $G(X')\geq G(X)$ then, by using
    the above monotonicity property of $f$ in $\lambda$ and the fact that $f(X',G(X'))=f(X,G(X))=0$,
    we have \[ \begin{aligned} G(X')-G(X) & \leq f(X',G(X'))-f(X',G(X)),\\ & \leq
    f(X,G(X))-f(X',G(X)).  \end{aligned} \] This inequality allows to transfer all the continuity
    properties of $f$ in $X$ to $G$ and we trust the reader to complete the proof by using this
    property.

Concerning the quasiconvexity property, first notice that for any $z \in\R$, with the notations as
    above, the sub-level sets of $f$ and $G$ satisfy \[ \{f(X,\lambda)\leq z\}=\{G(X)\leq
    z+\lambda\}\,, \] which implies that $f$ is quasiconvex (with respect to $p'$) if and only if
    $G$ is. Now, since the transforms $P(y,t,q), \mathcal{M}(y,t,q',M_T)$ are linear (and invertible) with respect to $(p',M_T)$ for fixed $q_N$, it
follows that the quasiconvexity property of $\tilde G$ implies the quasiconvexity of $f$, which
ends the proof.  \end{proof}

\

A final remark concerns the distance function which is classically used to build sub- and
supersolutions. Of course, it plays this role also here; but in order to be able to do so, the form
of $G$, namely \eqref{eqn:formG}, is essential and we point it out in the
\begin{lemma}\label{dist-func} Let $\psi:\R\to\R$ be a smooth, increasing function. Then, the
    function $w:=\psi(d)$ satisfies \[D^2_Tw(x)=\psi'(d(x))D^2d(x)\,.\] Moreover, if
\eqref{eqn:formG} and \eqref{ineq.gvbc} hold, then \[\G(x,t,Dw,D_T^2w)\leq \tilde G(x,t,0,0)-\bar
c\psi'(d(x))\,.\] \end{lemma}

\begin{proof} A straightforward computation shows that $(Dd(x)\otimes Dd(x))_T=0$ and $D_T ^2
    d(x)=D^2 d(x)$, which implies directly $D^2_Tw(x)=\psi'(d(x))D^2d(x)$. 

    Now, if $\psi'>0$, \eqref{eqn:formG} and \eqref{ineq.gvbc} hold, then \[ \begin{aligned}
        \G(x,t,Dw,D_T^2w) &:= \tilde G\Big(x,t,\psi'(d(x))Dd(x), \psi'(d(x))D^2 d(x)+\\ & \ \qquad
        \qquad \qquad \big(\psi'(d(x))Dd(x) \cdot n(x)\big) D^2d(x)\Big),\\ &= \tilde
        G(x,t,-\psi'(d(x))n(x), 0),\\[2mm] &\leq  \tilde  G(x,t,0, 0)-\bar c \psi'(d(x))\,.
    \end{aligned} \] In this computation, we used that $Dd(x)=-n(x)$ both for the gradient term and
the $D^2d(x)$ one, which disappears since $Dd(x)\cdot n(x)=-1$.  \end{proof}

This property allows to consider suitable choices of $\psi$ when building subsolutions. Of course, a
similar result holds for supersolutions when $\psi'<0$.

\section{Preliminaries}\label{br}

\def\Qxtrh{Q_{x,t}^{r,h}} \def\oQxtrh{\overline{Q_{x,t}^{r,h}}}

In this section we first list the exact hypotheses we are going to use in the sequel: on one hand,
we distinguish between ``basic assumptions'' which, in some sense, are the keystones of our
framework and, in particular, define what a Ventcell boundary condition is; on the other hand, we
have more specific assumptions which are required to obtain comparison results both in the cases
when $F$ is a first-order equation and when it is a second-order one. Then we devote several
subsections to preliminary results that are used later on.

\subsection{Hypotheses} \label{sec:hyp}

We begin with the assumption on $\Omega$ which is required in order to handle a Ventcell boundary
condition in a general  domain, see Section~\ref{GVbc-GD}. { Some of these assumptions may look a
little bit strange but the reason is rather simple: we need hypotheses on $F$ and $\G$ ensuring that
the nonlinearities obtained after the change of variables which is described in Section~\ref{RFB}
satisfy standard requirements.}

\begin{app}{\hyp{$\Omega$}}{Regularity of the domain.}{\smsp The (bounded or unbounded) domain
    $\Omega$ is of class $W^{4,\infty}$: there exists a bounded, $W^{4,\infty}$-function $d:\Omegb
    \to \R$ which agrees with the distance function in a neighborhood of $\domeg$ and such that
    $d(x)>0$ in $\Omega$ \footnote{Hence $d(x)=0$ iff $x\in \domeg$ and we recall that, if $x\in
\domeg$, $Dd(x)=-n(x)$ where $n(x)$ is the outward unit normal to $\domeg $ at $x$}.  } \end{app}

We point out that, for some results, the function $d$ being $C^2$, with bounded first and second
derivatives, is sufficient but to simplify matter, we only use \hyp{$\Omega$} in the paper. The
$W^{4,\infty}$-regularity is justified by the change of variable we perform in
Section~\ref{GVbc-GD}: we claim that the linear map $a(x)$ has the same regularity as $D^2\Psi$ and
has to be Lipschitz continuous. But $\Psi$ is built with $Dd$ and therefore the regularity of
$D^2\Psi$ cannot be better that the one of $D^3d(x)$, hence implying the $W^{4,\infty}$-regularity.

We then proceed with the standard hypotheses on the nonlinearities that are generally needed to use
the viscosity solutions' framework. To avoid repeating the same assumptions for $F$ and $G$---and to
point out that they are actually the same---, we introduce $H: A \times[0,T)\times \R^d \times
\mathcal{S}^d\to\R$ having in mind two cases \begin{enumerate} \item[(a)] $A=\Omega$, $d=N$ and
$H=F$; \item[(b)] $A=\domeg$, $d=N-1$ and $H=\G$.  \end{enumerate} We also use the notation
    $z=(x,t)$ with the usual distance $|z|^2=|x|^2+|t|^2$ and denote by $\lVert\cdot\rVert$ a
    matricial norm on $\mathcal{S}^d$.

The ``basic assumptions'' we mention above are

\begin{app}{\hyp{Gen}}{General assumptions on the Hamiltonians.} {\\ The nonlinearities $F,\G$ are
continuous functions and, with the above conventions (a)-(b), we have 
\begin{enumerate} 
    \item[$(i)$] \textit{Lipschitz continuity.}\\
        There exists a constant $C>0$ such that, for any $x\in
            A, t\in [0,T), p_1,p_2\in \R^d$ and $M_1,M_2\in \mathcal{S}^d$ 
            \[ \vert H(x,t,p_1,M_1)- H(x,t,p_2,M_2)\vert \leq C\Big(|p_1-p_2|+\lVert M_1 - M_2\rVert \Big).\]
    \item[$(ii)$] \textit{Degenerate ellipticity for the second-order case.} \\ For any $x\in A,
        t\in [0,T), p\in \R^d$ and $M_1,M_2\in \mathcal{S}^d$ \[ H(x,t,p,M_1)\leq H(x,t,p,M_2)
        \quad\hbox{if  }M_1 \geq M_2,\] where ``$\,\geq\,$'' denotes the partial ordering on symmetric matrices.  
\end{enumerate} 
Moreover, the function $\G=\G(x,t,p,M_T)$ has the form \eqref{eqn:formG} and 
\begin{enumerate} 
        \item[$(iii)$] There exists a constant $\bar c >0$ such that, for any
        $\lambda \geq 0$, $x\in \domeg$, $t\in [0,T)$, $p\in \R^N$ and $M\in \sym$ 
        \[\tilde G  \big(x,t,p+\lambda n(x),M_T\big)- \tilde G \big(x,t,p,M_T\big)\geq \bar c \lambda.\]
\end{enumerate}}
\end{app} We immediately point out that it is equivalent to say that $\G$ or
    $\tilde G$ satisfies \hyp{Gen}-$(i)$-$(ii)$. Now, of course Assumption~\hyp{Gen} is not
    sufficient to prove comparison results and we introduce the following (almost classical)
    assumption in which $B_A(0,R)$ denotes $B(0,R)\cap A$. Of course, we still use the above
    conventions (a)-(b) { and we remind the reader that he/she has to think about the change of
    variables of Section~\ref{RFB} to understand the modifications that we have to make on the
    standard hypothesis.}
   

\begin{app}{\hyp{Cont}}{Continuity assumption for the comparison result.}{\smsp For any $R,K>0$ and
    for any function $Q: \overline{B_A(0,R)}\times [0,T]\times \R^d \to \mathcal{S}^d$ such that,
    for any $z=(x,t),\tilde z=(\tilde x,\tilde t) \in \overline{B_A(0,R)}\times [0,T]$, $p\in \R^d$
    \[ \lVert Q(z,p)\rVert \leq K(1+|p|)\; , \; \lVert Q(z,p)-Q(\tilde z,p)\rVert \leq K|z-\tilde
    z|(1+|p|) ,\] there exists a modulus of continuity $\omega_{R,K}$ such that, for any
    $z,\tilde z\in \overline{B_A(0,R)}\times [0,T]$, $p\in\R^d$ and for any $X,Y \in
    \mathcal{S}^d$ satisfying \begin{equation}\label{ineqmat}
        \left[\begin{array}{cc}X&0\\0&-Y\end{array}\right] \le \frac1{\eps^2}
            \left[\begin{array}{cc}\Id&-\Id\\-\Id&\Id\end{array}\right]+
                \left[\begin{array}{cc}Q(z,p)&0\\ 0&-Q(\tilde z,p) \end{array}\right] +
                    \delta\left[\begin{array}{cc}\Id&0\\0&\Id\end{array}\right] \end{equation} for
                        some $\eps,\delta >0$, then we have \begin{equation}\label{cond:a31}
                        H(z,p,Y) - H(\tilde z,p,X) \le \omega_{R,K} \Big(|z-\tilde z|(1+|p|)+
                        \eps^{-2}|z-\tilde z|^2\Big)+ \omega_{R,K} (\delta).  \end{equation}}
    \end{app}

\vspace*{-5mm}

As a first remark, since $F$ can be a first-order equation, we remark that, in this case, this
    continuity hypothesis reduces to

\begin{app}{\hyp{Cont}}{First-order case.}{\smsp For any $R>0$ there exists a modulus of continuity
    $\omega_{R}:[0,+\infty\to [0,+\infty)$ such that, for any $z=(x,t),\tilde z=(\tilde x,\tilde
    t)\in \overline{B_{\Omegb}(0,R)}\times [0,T]$, $p\in\R^d$, \begin{equation}\label{cond:a31-p}
        F(z,p) - F(\tilde z,p) \le \omega_{R} \Big(|z-\tilde{z}|(1+|p|)\Big)\,.  \end{equation}}
\end{app}

\vspace*{-5mm}

In the classical case, the $Q$-term in Hypothesis~\hyp{Cont} does not exist; here it comes from the
    change of coordinates we perform in a neighborhood of the boundary and therefore appear only in
    the second-order case, \cf \eqref{cond:a31-p}; therefore, this term is needed only in such
    neighborhood. In order to keep things as simple as possible, we do not try to generalize this
    assumption to take this remark into account.

In fact, this assumption as the classical one is satisfied by Hamilton-Jacobi-Bellman or more
    generally by Isaacs Equations under standard assumptions, namely if $F$ is given by \[
        \sup_{\alpha} \inf_{\beta}\Big\{ -{\rm Tr}(a(x,t,\alpha,\beta)M)-b(x,t,\alpha,\beta)\cdot p
        -f(x,t,\alpha,\beta) \Big\},\] where $a=\sigma (x,t,\alpha,\beta)\,\trp{\sigma}
    (x,t,\alpha,\beta)$, the functions $\sigma (x,t,\alpha,\beta)$ and $b(x,t,\alpha,\beta)$ being
    bounded, locally Lipschitz continuous in $(x,t)$ uniformly w.r.t. $\alpha,\beta$ and $f
    (x,t,\alpha,\beta)$ is continuous in $(x,t)$ uniformly w.r.t. $\alpha,\beta$. For $\G$, we may
    take into account nonlinearities given by similar and properly adapted formulas.

Now we introduce some specific requirements on $F$ in the normal direction to the boundary. These
    are not the same according to the first or second order case. These conditions will play a
    crucial role and in order to get a comparison result---we refer to the book of the
    authors~\cite{BCbook} for detailed explanations on the role of the normal coercivity in the
    first-order case. In the second-order case, the ingredient that replaces the coercivity is the
    normal strong ellipticity as will be clear in the comparison proof below.

We give these two assumptions in the ``flat case'', i.e. when $\Omega:=\{x_N>0\}$, but one can
    translate them in a straightforward way for the general case.


\begin{app}{\hyp{NC}}{Normal coercivity and quasiconvexity, first-order case.}{\smsp {   For any
    $(x,t)\in \domeg \times [0,T]$, the function $p\mapsto F(x,t,p)$ is quasiconvex and there exists
    $r,\bar \eta, \bar C >0$ such that, for any $(y,s)\in \Omegb$ satisfying $|y-x|+|s-t|<r$, $p'\in
    \R^N$ such that $p'\cdot e_N=0$, $\lambda \in \R$, \[ F(y,s,p'+\lambda e_N) \geq \bar \eta
|\lambda| -\bar C(1+|p'|).\]}} \end{app} 

Still for the first-order case, we require some
quasiconvexity assumption on $G$:

\begin{app}{\hyp{QC-G}}{Quasiconvexity of the Ventcell boundary condition.}{\smsp For any 
    $(x,t,p')\in \domeg \times [0,T]\times \R^{N-1}$, the function $M_T \mapsto G(x,t,p',M_T)$ is quasiconvex.}
\end{app}

We point out that both quasiconvexity assumptions on $F$
and $G$ can be stated equivalently either in the general case or in the flat boundary case by the
same argument used in the proof of Lemma~\ref{V-flat}.

Finally we impose the following assumption on $F$ in the second-order case:

\begin{app}{\hyp{NSE}}{Normal strong ellipticity, second-order case.}{\smsp For any $(x,t)\in \domeg
    \times (0,T)$ there exists $r,\bar \eta, \bar C >0$ such that, for any $(y,s)\in \Omegb$
    satisfying $|y-x|+|s-t|<r$, $p\in \R^N$, $M\in \sym$ and $\lambda \in \R$, \[ F(y,s,p, M+\lambda
    e_N\otimes e_N) \leq -\bar \eta \lambda +\bar C(1+|p|+|M|) \quad \hbox{if  }\lambda >0,\] \[
        F(y,s,p, M+\lambda e_N\otimes e_N) \geq -\bar \eta \lambda -\bar C(1+|p|+|M|) \quad \hbox{if
        }\lambda <0.\] } \end{app}

\vspace*{-7mm}

\begin{remark}\label{rem:Q.lip} Let us come back on the local Lipschitz continuity both in $x$ and $t$ we
    impose in \hyp{Cont}, \cf also \eqref{cond:a31-p}. The reader may think that this requirement is
    not natural; one may just expect some continuity in $t$.  However, in order to use efficiently
    \hyp{NC} in the first-order case, we need the variable $t$ to be considered as a tangential
variable $x'$, thus imposing the same regularity on both---see the proof below. In the second-order
case, though the situation is different, we still use this common regularity for some technical
reason.  \end{remark}

We can now sum up the requirements on the equation in both the first and second-order case as well
as for the boundary condition for comparison results.

\newpage

\begin{app}{\hyp{Comp-1}}{Assumptions on $F,\G$ in the first-order case.}{\smsp The nonlinearities
    $F,\G$ satisfy Assumptions~\hyp{Gen}, \hyp{Cont}\footnote{which reduces to \eqref{cond:a31-p}
    for $F$.}, the normal coercivity assumption \hyp{NC} holds for $F$ and the quasiconvexity
    assumption \hyp{QC-G}\footnote{Or their equivalent properties in the non-flat case.}.} \end{app}

\begin{app}{\hyp{Comp-2}}{Assumptions on $F,\G$ in the second-order case.}{\smsp The nonlinearities
$F,\G$ satisfy Assumptions~\hyp{Gen}, \hyp{Cont} and the normal strong ellipticity assumption
\hyp{NSE} holds for $F$.} \end{app}

These assumptions on $\G$ mean that the associated ``flat boundary condition'' $\tilde G$ has to
satisfy first the standard second-order assumptions (ellipticity and Lipschitz continuity), but also
the Neumann or Ventcell-type boundary condition already mentionned in Section~\ref{GVbc-GD}.


\subsection{Global Comparison Results from Local Comparison Results} \label{sec:GCR.LCR}

In \cite{BCbook}, the proof of a ``global'' \SCR is reduced to the simpler proof of a ``local one'',
and even to the proof of a \SCR in a small ball.  Here we follow the same strategy since it is
well-adapted to problems with boundary conditions and, in order to emphasize the difference
``global--local'', we denote by \GCR a ``Global Comparison Result'' while a ``Local Comparison
Result'' is denoted by \LCR. Here are the definitions of these two types of results.


\begin{app}{\GCR}{Strong (global) Comparison Result for \eqref{Eqn}-\eqref{Vbc-GC}.}{\smsp If
    $u:\Omegb \times [0,T) \to \R$ is a bounded upper semicontinuous subsolution of
    \eqref{Eqn}-\eqref{Vbc-GC}, if $v:\Omegb \times [0,T) \to \R$ is a bounded lower semicontinuous
    supersolution of \eqref{Eqn}-\eqref{Vbc-GC} and if $u(x,0) \leq v(x,0)$ in $\Omegb$, then
$u(x,t) \leq v(x,t)$ in $\Omegb\times [0,T).$} \end{app}

\bigskip

In \cite{BCbook}, it is shown that, under suitable conditions, the proof of a \GCR can be reduced to
the proof of a \LCR. In order to give a precise definition of a \LCR, we introduce the notations \[
    \Qxtrh:=\{(y,s)\in \Omegb \times [0,T) : |y-x|<r,\ t-h<s<t\} ,\] \[ \partial_p
    \Qxtrh:=\{(y,s)\in \oQxtrh: |y-x|=r\}\cup \{(y,s)\in \oQxtrh: s=t-h\}.\]

\begin{app}{\LCR}{Local Comparison Result for \eqref{Eqn}-\eqref{Vbc-GC}.}{\smsp For any $(x,t) \in
    \Omegb \times (0,T)$, there exists $\bar r, \bar h>0$ such that\\ { if
    $u:\overline{Q_{x,t}^{\bar r,\bar h}} \to \R$ is a bounded upper semicontinuous subsolution of
    \eqref{Eqn}-\eqref{Vbc-GC} in $Q_{x,t}^{\bar r,\bar h}$,} \\ if $v:\overline{Q_{x,t}^{\bar
    r,\bar h}} \to \R$ is a bounded lower semicontinuous supersolution of \eqref{Eqn}-\eqref{Vbc-GC}
in $Q_{x,t}^{\bar r,\bar h}$,\\[2mm] then, for any $0<r\leq \bar r$ and $0<h\leq \bar h$, \[
\max_{\oQxtrh}(u-v)_+ \leq \max_{\partial_p \Qxtrh }(u-v)_+ .\] } \end{app}

Our result is the (notice that in the result below, of course \hyp{Gen}-$(ii)$ is automatically
satisfied if $F$ is a first-order Hamiltonian)

\begin{proposition}\label{lcr-scr} Assume that \hyp{$\Omega$} holds, and that both $F$ and $\G$
satisfy \hyp{Gen}. Then \LCR implies \GCR.
\end{proposition}

Before providing the proof of Proposition~\ref{lcr-scr}, we introduce a family of functions which
will be used in several places throughout this article, in particular to take care of the Ventcell
boundary condition: for $K>0$, we select a function $\varphi_K:[0,+\infty)\to\R$ satisfying
\begin{enumerate} 
    \item $\varphi_K\in C^2([0,+\infty),\R)$, decreasing; 
    \item $\varphi_K(0)=0,\
        \varphi_K'(0)=-1,\ \varphi_K''(0)=-K$; 
    \item $\varphi_K'$ has a compact support, more
        precisely $\mathrm{supp}(\varphi'_K)=[0,1]$; \item In particular, $\varphi_K$ is constant
            for $t\geq 1$ and therefore $\varphi_K$ is bounded.  
\end{enumerate}

\begin{proof} We slightly modify the arguments of \cite{BCbook} in order to take into account the
    Ventcell boundary condition. We denote by $u$ and $v$ the bounded sub and supersolution to be
    compared.

    We first have to localize problem and to do so, we introduce the function 
    \[\chi(x,t):=(|x|^2+1)^{1/2} + k_1
    \varphi(d(x)) +\frac{k_2}{T-t} ,\] where $\varphi=\varphi_1$ is defined just above, with $K=1$ 
    ($K$ is not going to play any role in this proof). Using
    \hyp{Gen}-$(i)$ and $(iii)$ together with Lemma~\ref{dist-func}, one easily shows that, by
    choosing $k_1$ large enough and then $k_2$ large enough, then $u_\alpha(x,t)=u(x,t)-\alpha
    \chi(x,t)$ is still a subsolution for \eqref{Eqn}-\eqref{Vbc-GC} for any $\alpha>0$ and
    $u_\alpha(x,t)\to -\infty$ when $|x|\to +\infty$ uniformly with respect to $t$.

    The aim is to show that $u_\alpha\leq v$ on $\Omegb \times [0,T)$ for any $\alpha>0$; indeed, if
    this is true, we obtain the \GCR by letting $\alpha$ tend to $0$.

    Because of the behavior of $u_\alpha$ at infinity, the maximum of $u_\alpha-v$ is achieved at
    some point $(x,t)$ and we can choose $t$ as the minimal time for which this maximum is achieved.
    Of course, we can assume without loss of generality that $t>0$, otherwise we are done, and then
    we face two cases: either $x\in \Omega$ or $x\in \domeg$.

    \noindent\textbf{1.} If $x\in \Omega$, the arguments of \cite{BCbook} apply: we argue in $\Qxtrh$ where $r,h$ are
    chosen small enough in order that the \LCR holds; we will also choose $h>0$ small compared to
    $r$, its size will be made precise later on. Notice that we can choose $r,h$ such that $\Qxtrh$
    does not intersect $\domeg \times (0,T)$ and $t-h\geq 0$.
    
    For $k_3>0$ large enough, $u_\alpha^\delta(y,s):=u_\alpha(y,s) -\delta(|y-x|^2 +k_3(s-t))$ is still
    a subsolution of \eqref{Eqn} and, if $0<h\ll r^2$, the function $|y-x|^2 + k_3(s-t)$ is strictly
    positive on the lateral boundary $\Lambda_{x,t}^{r,h}:=\{(y,s):|y-x|=r,\ t-h\leq s\leq t\}$; indeed 
    \begin{equation}\label{lat.ineq}
        |y-x|^2 + k_3(s-t)=r^2+k_3(s-t)\geq r^2-k_3h>0 \quad\hbox{provided  }h < r^2/k_3.
    \end{equation} 
    Next, for $s=t-h$, the maximum of $u_\alpha-v$ cannot be achieved by the minimality of
    $t$ and thus, by choosing $\delta$ small enough, we have 
    \[ u_\alpha^\delta (x,t)-v(x,t) > \max_{|y-x|\leq r}(u_\alpha^\delta (y,t-h) -v(y,t-h)). \]
    In other words, the maximum of $u^\delta_\alpha-v$ on $\partial_p Q_{x,t}^{r,h}$ is
    attained on the lateral boundary $\Lambda_{x,t}^{r,h}$ which is defined above.
    Applying the \LCR to
    $u_\alpha^\delta$ and $v$ and taking into account the above pieces of information, we have
    \[ 
    \begin{aligned} u_\alpha (x,t)-v(x,t) &=  u_\alpha^\delta (x,t)-v(x,t)\\ 
        & \leq  \max_{\partial_p \Qxtrh}(u_\alpha^\delta (y,s)-v(y,s)),\\ 
        & \leq  \max_{\Lambda_{x,t}^{r,h}}(u_\alpha^\delta (y,s)-v(y,s)),\\ 
        & < \max_{\Lambda_{x,t}^{r,h}}(u_\alpha(y,s) -v(y,s)),
    \end{aligned}
    \]
    which yields a clear contradiction with the definition of $(x,t)$,
    and completes the proof in this case.

    \noindent\textbf{2. }In the case when $x\in \domeg$, the advantage of reducing the proof to a \LCR, and therefore to
    a small ball around $x$, is that we can argue w.l.o.g. with a flat boundary, \ie in the case of
    \eqref{Vbc-HS}, \cf Lemma~\ref{V-flat}. Even if this requires a few additional arguments---in
    particular, the change of coordinates does not transform balls into balls---we trust the reader
    to be able to convince him/herself of this fact.
    
    With this reduction, this second case is treated analogously by adding an extra term to take care
    of the Ventcell condition, namely replacing the $\delta$-term by 
    \[ \delta\Big( |y-x|^2 + k\eta \varphi(x_N/\eta)+ k_3\big(s-t)\big) \Big)\,.\]
    Here $\varphi=\varphi_1$, \ie $K=1$ in the definition of $\varphi_K$ given above the
    proof, and the three parameters $\eta, k, k_3 >0$ have to be chosen in a such way that:
    $(i)$ $k$ is large enough ; $(ii)$ $k\eta$ is small enough ; $(iii)$ $k_3$ is large compared to $k/\eta$.
    
    Indeed, using the properties of $\varphi$, the derivative of the $\varphi$-term is $-k$ if $x_N=0$,
    \ie if $x\in \domeg$ and, by choosing  $k>0$ large enough, the $\delta$-term has a negative
    contribution in
    \eqref{Vbc-HS}, this property being true independently of $\eta$ and $k_3$.
    Then, we notice that the second order spatial derivatives of the $\delta$-term is of the order of
    $k/\eta$. So, if $k_3>0$ large enough compared to $k/\eta$, $u_\alpha^\delta$ is still a subsolution
    of \eqref{Eqn}-\eqref{Vbc-HS}.

    It remains to choose $\eta$. For $\eta>0$ small enough, $k\eta \varphi(x_N/\eta)=O(k\eta)$ is
    negative but small compared to $r^2$, which yields a contradiction on the lateral boundary
    $|y-x|=r$ as in \eqref{lat.ineq}. Next, on the boundary $s=t-h$, taking $\delta$ small enough gives the answer since,
    again, by the minimality of $t$, the maximum of $u_\alpha -v$ is strictly less than
    $u_\alpha^\delta (x,t)-v(x,t)$ for $s=t-h$. Again, the contradiction is obtained for $\delta>0$
    small enough, and the proof is complete.
     \end{proof}

\subsection{Local Properties of the Ventcell Boundary Condition}\label{lpvc}

    As the title indicates it, we investigate the {\em local} properties of the Ventcell boundary
    condition and therefore we may assume without loss of generality that $\Omega=\{x_N>0\}$ and
    that we are in the case of \eqref{Eqn}-\eqref{Vbc-HS}.
    
    The first result is rather classical. However it contains two important pieces of 
    information: on one hand, the result states that, under suitable condition, the sub and supersolutions
    properties can be extended up to the terminal time and, on the other hand, it proves the 
    ``regularity'' (in the sense of \cite{BCbook}) of sub and supersolutions in the case of Ventcell
    boundary conditions. In order to obtain the regularity result, we use either the normal coercivity
    of the nonlinearity in the case of a first-order equation or the normal strong ellipticity in
    the case of a second-order one.

Our result is the
\begin{proposition}\label{regul} 
    Assume that $F, G$ satisfy \hyp{Gen} and let $u$ be a bounded \usc subsolution of
    \eqref{Eqn}-\eqref{Vbc-HS} and $v$ be a bounded \lsc supersolution of
    \eqref{Eqn}-\eqref{Vbc-HS}.\\[2mm]
        $(i)$ For any $0<t<T$, $u$ and $v$ are respectively a sub and a supersolution of
            \eqref{Eqn}-\eqref{Vbc-HS} on $\Omega \times (0,t]$.\\[2mm]
        $(ii)$ Assume in addition that $F$ satisfies either \hyp{NC} or \hyp{NSE}.  Then the
            functions $u$ and $v$ are regular on
            $\domeg \times (0,T)$. More precisely, for any $(x,t) \in \domeg \times (0,T)$, 
            \[ u(x,t)=\limsup_{{\substack{ (y,s)\to (x,t)\\ (y,s)\in \Omega \times (0,t]}}} u(y,s)\;,\quad 
            v(x,t)=\liminf_{{\substack{ (y,s)\to (x,t)\\ (y,s)\in \Omega \times (0,t]}}} v(y,s). 
            \] 
\end{proposition}

The regularity property $(ii)$ is stronger than the usual one: not only does it mean that the value
of $u$ and $v$ on the boundary are, in some sense, the limit of their interior values (\ie there is
no artificial jump on the boundary) but these interior values also correspond to values of $u$ and $v$ at
previous times. And, of course, the same general result holds in general domains.

\begin{proof} We first prove $(i)$ and the arguments being similar in the sub and supersolution cases,
we just give them in the subsolution one.

We remind the reader that we have to show that the expected subsolution inequality has to be satisfied
if $(x,t)$ is a local strict maximum point of $u-\phi$ on $\Omegb \times (0,t]$, not on $\Omegb \times 
(0,T)$. We just treat the case when $x\in \domeg$, the case when $x\in \Omega$ being far easier (and
is a standard result in the viscosity literature). 

For $0<\e \ll 1$, we consider the function
$$ \Psi_\e (y,s):= u(y,s)-\phi(y,s)- \frac{[(s-t)^+]^2}{\e^2}.$$
Since $(x,t)$ is a local strict maximum point of $u-\phi$ on $\Omegb \times (0,t]$, this function has a local 
maximum point---relatively to $\Omegb \times 
(0,T)$---denoted by $(y_\e, s_\e)$, near $(x,t)$. Moreover, $(y_\e, s_\e)\to (x,t)$ when $\e \to 0$.

    The subsolution inequality holds at $(y_\e, s_\e)$, and only one additional term appears, in
    \eqref{Eqn}, and not in \eqref{Vbc-HS}, namely the time-derivative $$ \frac{2(s_\e-t)^+}{\e^2}\,.$$
    This term being positive, it can be dropped in the subsolution inequality and 
    the result then follows by letting $\e$ tend to $0$. Of course, we can also drop this term in
    the supersolution inequality since, in this case, it comes with a minus sign.

    Now we prove $(ii)$ and again only for the subsolution $u$.
    We assume by contradiction that there exists $(x,t) \in \domeg \times (0,T)$ such that
    \begin{equation}\label{u-nonreg} 
        u(x,t)>\limsup_{{\substack{ (y,s)\to (x,t)\\ (y,s)\in \Omega \times (0,t)}}} u(y,s), 
    \end{equation} 
    and the aim is to get a contradiction. 
    
    To do so, for $0<\e\ll 1$, we introduce the function
    \[ \chi_\e (y,s) := u(y,s)-
    \frac{|y-x|^2}{\e^2}-\frac{|s-t|^2}{\e^2}- \e \varphi (L y_N) - \e \varphi(L^2(t-s))\,,\] 
    where $L$ is a positive constant to be chosen later on and $\varphi=\varphi_K$ with $K=1$, where
    the functions $\varphi_K$ are defined just after Proposition~\ref{lcr-scr}.

For $\e$ small enough and for any $L>0$, this function has a maximum point on $\Omegb 
\times (0,t]$ near $(x,t)$; we denote it by $(\xe,\te)$. If \eqref{u-nonreg} holds, then this maximum point is 
necessarily either on  the boundary $\domeg \times (0,t]$ or on $\Omega \times \{t\}$. Indeed, this follows from the fact that (i) $\chi_\e(x,t)=u(x,t)$, (ii) $\varphi$ being bounded, the two last terms are of order $\e$
and (iii) the two penalization terms $ \frac{|y-x|^2}{\e^2}-\frac{|s-t|^2}{\e^2}$ ensure that $(\xe,\te)$ is close to $(x,t)$.

Now we choose $L \gg \e^{-3}$ and we examine the two possible cases:\\[2mm]
    $\bullet$ If $(\xe,\te)\in \Omega \times \{t\}$, in the subsolution inequality we get that
    the derivative of the last term in $\chi_\e$ yields a positive contribution; more precisely,
    since $t_\e=t$, the contribution is $-L^2\e\varphi'(0)=L^2\e$.  

    On the other hand, the two first terms in $\chi_\e$ provide an $\e^{-2}$ contribution and the
    third one has a first derivative of order $L\e$, but is concave---hence the second derivative of this term
    has a positive contribution in the equation. Taking into account the
    properties of $F$, we conclude that the positive $L^2\e$ contribution dominates the equation
    provided $L^2\e\gg L\e$ and $L^2\e\gg \e^{-2}$. Therefore, if $L\gg\e^{-2}$ for instance, then the subsolution
    inequality cannot hold for $\e>0$ small enough, which provides a contradiction.\\[2mm]
$\bullet$ If $(\xe,\te)\in \domeg \times (0,t]$, the derivative of the $\e \varphi (L y_N)$-term gives a 
positive $L\e$ contribution which, under the assumptions on $G$, dominates the $\e^{-2}$ contribution of
the two first terms in $\chi_\e$ if $L \gg \e^{-3}$ and this provides again a contradiction.

Hence \eqref{u-nonreg} cannot hold and we have proved the desired regularity property.
\end{proof}

The next result concerns the boundary condition for second-order equations which satisfy hypothesis
\hyp{NSE}, \ie which are uniformly elliptic in the normal direction; in this case, the Ventcell
boundary condition holds in a strong sense.

\begin{proposition}\label{VBC-uec} Assume that $F,G$ satisfy \hyp{Gen} and that $F$ satisfies
    \hyp{NSE}. Then the Ventcell boundary condition is satisfied in a ``strong sense'' for both
    subsolutions and supersolutions of \eqref{Eqn}-\eqref{Vbc-HS}. More precisely,\\[2mm] $(i)$ if
    $u$ is an \usc subsolution of \eqref{Eqn}-\eqref{Vbc-HS} and $(x,t) \in \domeg \times (0,T)$ is
    a local maximum point of $u-\phi$, where $\phi$ is a smooth test-function then \[
        -\frac{\partial \phi}{\partial x_N} (x,t) + G(x,t,\Dxp  \phi(x,t) , \Ddxp   \phi(x,t)
        )\leq 0\,.\] \noindent $(ii)$ if $v$ is a \lsc  supersolution of \eqref{Eqn}-\eqref{Vbc-HS}
    and $(x,t) \in \domeg \times (0,T)$ is a local minimum point of $v-\phi$, where $\phi$ is a
smooth test-function then \[ -\frac{\partial \phi}{\partial x_N} (x,t) + G(x,t,\Dxp  \phi(x,t) ,
\Ddxp  \phi(x,t) )\geq 0\,.\] \end{proposition}

\begin{proof} We sketch the proof for the subsolution case, the supersolution one being analogous.

    If $(x,t) \in \domeg \times (0,T)$ is a local maximum point of $u-\phi$, it is also a local
    maximum point of the function \[ (y,s) \mapsto u(y,s)-\phi(y,s)- \delta y_N+Ly_N^2
    ,\] for any $\delta,L>0$. Of course, the ``locality'' in this property depends on $\delta$ and
    $L$.  The second-derivative of the new test-function at $(x,t)$ is now \[ D^2 \phi(x,t)
    -2Le_N\otimes e_N  ,\] and, using \hyp{NSE}, it is clear that, for $L$ large enough, the
    $F$-inequality cannot hold and therefore \[ -\frac{\partial \phi}{\partial x_N} (x,t) -\delta +
    G(x,t,\Dxp  \phi(x,t) , \Ddxp   \phi(x,t) )\leq 0.\] Letting $\delta$ tend to $0$ gives the
result.  \end{proof}

\subsection{About the initial condition}

A last property concerns the initial data and more precisely the points of $\domeg \times \{0\}$.
If \eqref{Eqn}-\eqref{Vbc-HS} is associated to the initial data \begin{equation}\label{init}
u(x,0)=u_0(x) \quad \hbox{on  }\Omegb  , \end{equation} where $u_0 \in C(\Omegb)$, then a priori we
have to use ``initial data in the viscosity solutions sense'' in the same way as we have ``boundary
conditions in the viscosity solutions sense''. This is the requirement to be able to apply the
half-relaxed limit method in its full powerness. {By standard methods, one can prove that,} if $u$
is a subsolution of \eqref{Eqn}-\eqref{Vbc-HS}-\eqref{init} and $v$ is a supersolution of
\eqref{Eqn}-\eqref{Vbc-HS}-\eqref{init}, we have \begin{equation}\label{ineg-init} u(x,0)\leq u_0(x)
\leq v(x,0)\quad \hbox{for any  }x\in \Omega.  \end{equation} But we have to show that this
inequality still holds if $x\in \domeg$, which is the aim of the
\begin{proposition}\label{prop-init} Assume that $F,\G$ satisfy \hyp{Gen} and that $u_0 \in
C(\Omegb)$.  Then \eqref{ineg-init} holds for any $x\in \Omegb$.  \end{proposition}

\begin{proof} We only prove the result for a subsolution $u$, the proof for a supersolution being
    analogous.  And of course, we consider a point $x\in \domeg$ for which we want to show that
    $u(x,0)\leq u_0(x)$.

    For $\e$ small enough and for some large enough constant $K_1>0$ to be chosen later on, we
    consider the function \[ (y,t) \mapsto u(y,t)-\frac{|y-x|^2}{\eps^2}-K_1t -
    \eps\varphi(\frac{y_N}{\eps^4}) ,\] in the compact set $(\overline{B(x,1)}\cap \Omegb) \times
    [0,T]$ where $\varphi=\varphi_1$ defined at the beginning of Section~\ref{sec:GCR.LCR}. This
    function achieves its maximum at $(x_\eps,t_\eps)$ and, using that the $\eps\varphi$-term tends
    to $0$, classical arguments allow to show that \[ \frac{|x_\eps-x|^2}{\eps^2}\to 0\quad\hbox{as
    }\eps \to 0.\] In particular, for $\e$ small enough, $x_\eps\in B(x,1)\cap \Omegb$---it is not
    on the boundary of the ball---and we can write down viscosity subsolution inequalities. We claim
that, for $\e$ small enough and for $K_1>0$ large enough,  we have necessarily $t_\eps=0$ and
$u(x_\eps,0)\leq u_0(x_\eps)$. Indeed \begin{enumerate} \item[$(i)$] On one hand, if $\e$ is small
            enough, the Ventcell boundary condition cannot hold since the $\eps\varphi$-term has a
        derivative which is $+\e^{-3}$ while all the $x'$-derivatives at at most of order $\e^{-2}$.
    \item[$(ii)$] On the other hand, if $K_1$ is large enough (of order, say, $\eps^{-8}$), the
        equation cannot hold either.  \end{enumerate} Hence only the inequality associated to the
initial data can hold, proving our claim. To conclude, it suffices to recall that $u_0$ is
continuous and $u(x_\eps,0)\to u(x,0)$ invoking again classical arguments.  \end{proof}

Again, in the result above, of \hyp{Gen}-$(ii)$ is automatically satisfied if $F$ is a first-order
Hamiltonian.

\subsection{Reduced parabolic sub- and superjets}
\label{sec:jets}

Here for simplicity of notations we denote by $\Qbl:=\bar\Omega\times(0,T)$ and
$\pQbl:=\partial\Omega\times(0,T)$.  We introduce a \textit{reduced} version of the parabolic
variants of the second-order sub- and superjets given in the ``User's guide'' of Crandall, Ishii and
Lions~\cite{Users}. Indeed, we are just interested here in Problem~\eqref{Eqn}-\eqref{Vbc-HS} for
which we have to give a weak sense to $u_t, Du$ and $\Ddxp  u$ only. In all the following, the
subscript ``$r$'' refers to ``reduced''.

\medskip

\noindent\textbf{Reduced semijets ---} The reduced parabolic superjet $\PP^{2,+}_r u(x,t)$ of an \usc
function $u:\Qbl\to\R$ at a point $(x,t)\in\Qbl$ is defined as the set of all
$(p_x,p_t,M)\in\R^N\times\R\times \mathcal{S}^{N-1}$ such that
\[ 
    \begin{aligned}
        u(y,s)\leq u(x,t) &+ p_x\cdot(y-x)+p_t(s-t) + \frac12(M(y'-x'),y'-x') \\
        &+ o(|y'-x'|^2+ |y_N-x_N|+|s-t|)\,.
        \end{aligned}
\]
Similarly, the reduced parabolic subjet $\PP^{2,-}_r v (x,t)$ of a \lsc function $v$ is given by
\[ \PP^{2,-}_r v (x,t)=-\PP^{2,+}_r (-v)(x,t)\,. \]

In several places, we even go a step further by considering only the $(p_{x_N},M)$ components
when only those terms play a role in the equation, related to the $x_N$-derivative and second derivatives
\wrt $x'$. In other words, the couples corresponding to $(D_{x_N},\Ddxp  )$. We still keep the
notation $\PP^{2,+}_r,\PP^{2,-}_r$ in this case.

\medskip

\noindent\textbf{On the structure of semijets ---} The following result is a slightly weaker
adaptation of \cite[Prop. 2.10]{BCbook} for reduced second-order parabolic superjets.
\begin{proposition}\label{prop:jets}
    Assume that $\Omega$ is given by \eqref{Om-HS}, that \hyp{Comp-1} holds and let $u$ be a subsolution of
    Problem~\eqref{Eqn}-\eqref{Vbc-HS}. If $(p_x,p_t,M)\in \PP^{2,+}_r u(\xb,\tb)$ at a point $(\xb,\tb)\in\pQbl$,
    then
    \begin{equation}
        \Lambda^+(u):=\big\{\lambda\in\R: (p_x+\lambda e_N,p_t,M)\in \PP^{2,+}_r u(\xb,\tb)\big\},
    \end{equation}
is an interval, either of the form $(\lambda_{min},+\infty)$ or $[\lambda_{min},+\infty)$ with $\lambda_{min}\in \R$.
    Moreover, if $\underline\lambda$ is the maximal solution of
    \begin{equation}\label{eq:prop-u-lambda}
     p_t+F(\xb,\tb, p_x +\underline \lambda e_N)=0\,,
    \end{equation}
    then $(\underline\lambda,+\infty)\subset\Lambda^+(u)$.
\end{proposition}
\begin{proof}
    It is clear that the set $\Lambda^+(u)$ is convex and nonempty since $\lambda=0$ belongs to
    $\Lambda^+(u)$. Moreover, if $\lambda \in \Lambda^+(u)$ and if $\lambda' \geq \lambda$, it is
    clear from the definition that $\lambda' \in \Lambda^+(u)$ since $(\xb,\tb)\in\pQbl$. Hence
    $\Lambda^+(u)$ is an interval which has one of the two announced forms.  The main difference
    with the case of first-order sub- and superdifferentials is that $\PP^{2,+}_r u(\xb,\tb)$---and
    therefore $\Lambda^+(u)$---may not be closed\footnote{Take $u(x):=x_N^{3/2}$ then it is easy to
    check that $\Lambda^+(u)=(0,+\infty)$ for any point $(\xb,\tb)\in\pQbl$ and that $\lambda =0$ is
    not associated to any point in $\PP^{2,+}_r u(\xb,\tb)$.}.
    
    Moreover, if $(p_x+\lambda e_N,p_t,M)\in \PP^{2,+}_r u(\xb,\tb)$ and if we write $p_x=(p',p_N)$ with $p'\in \R^{N-1}$ and $p_N\in \R$, the subsolution inequality reads
    $$ \min(p_t+F(\xb,\tb, p_x + \lambda e_N), -\lambda - p_N+ G(\xb',\tb, p', M))\leq 0,$$
which clearly implies that $\lambda$ is bounded from below and therefore $\lambda_{min}\in \R$.

    Since $(p_x,p_t,M)\in \PP^{2,+}_r u(\xb,\tb)$, there exists a test-function $\varphi\in
    C^2(\Qbl)$ such that $u-\varphi$ reaches a maximum at $(\xb,\tb)$ and such that
    $$\varphi_t (\xb,\tb)=p_t, \;D_x\varphi(\xb,\tb)=p_x, \;\Ddxp  \varphi
    (\xb,\tb)=M.$$
    Moreover, we can assume that this maximum is strict, by standard arguments. 

    Then, for $\lambda>\underline{\lambda}$ and $0<\beta\ll 1$, we consider the function
    \[ \zeta_\beta (x,t)=u(x,t)-\varphi(x,t) -\lambda x_N -\frac{x_N^2}{\beta}\,. \]
    Since $(\xb,\tb)$ is a strict local maximum point of $u-\varphi$, there exists a sequence 
    $(x_\beta, t_\beta)\in\Qbl$ of local maximum points of $\zeta_\beta$ which converges to $(\xb,\tb)$.

    From now on, we drop the indices $\beta$ for the sake of notational simplicity and just denote $(x_\beta, t_\beta)$ by $(x,t)$. If $x_N>0$, then
    we would have 
    \[ \varphi_t (x,t)+F\big(x,t, D_x\varphi(x,t) +(\lambda+ \frac{2x_N}{\beta}) e_N\big)
    \leq 0.\]
    But, using \hyp{NC}, the term $(\lambda+ \frac{2x_N}{\beta})$ remains bounded and therefore by
    continuity, this would mean
    \[  p_t +F\big(\xb,\tb, p_x +(\lambda+ \frac{2x_N}{\beta}) e_N\big)+o_\beta(1) \leq 0\,. \]
    And we reach a contradiction for $\beta$ small enough since, by definition of
    $\underline\lambda$, for any $\mu\geq \lambda >\underline\lambda$, we have,
    \[ p_t +F\big(\xb,\tb, D_x\varphi(\xb,\tb) + \mu e_N)\geq \eta(\lambda)>0\,. \]
    Hence $x_N=0$ and therefore $(x,t)=(\xb,\tb)$ by the strict maximum point property. We deduce
    that  
    \[ (D_x\varphi(\xb,\tb)+\lambda e_N, \varphi_t (\xb,\tb), \Ddxp  \varphi
    (\xb,\tb))=(p_x+\lambda e_N,p_t,M)\in \PP^{2,+}_ru(\xb,\tb)\,. \]
    Hence $(\lambda,+\infty)\subset \Lambda^+(u)$ and the result follows.
\end{proof}

Of course this result can be translated into a similar one for the reduced subjet of a supersolution
$v$, through the formula $\PP^{2,-}_rv=-\PP^{2,+}_r(-v)$.

\medskip

\noindent\textbf{Closures of semijets ---}
In order to apply then Jensen-Ishii Lemma, as in the User's guide of Crandall, Ishii and
Lions~\cite{Users}, we define the sets $\bar \PP^{2,+}_r \bar u^\alpha  (x,t)$ and $\bar \PP^{2,-}_r \bar
v_\alpha (y,s)$ by the following way: we say that $(\lambda,X) \in \bar \PP^{2,+}_r \bar u^\alpha (x,t)$ if and only if
there exists a sequence $(x_k,t_k,\lambda_k,X_k)_k$ converging to $(x,t,\lambda,X)$ such that, for all $k\in\N$, $(\lambda_k,X_k)\in \PP^{2,+}_r \bar u^\alpha  (x_k,t_k)$.

The definition of $\bar \PP^{2,-} \bar v_\alpha (y,s)$ is analogous, replacing $\PP^{2,+} \bar
u^\alpha (x_k,t_k)$ by $\PP^{2,-} \bar v^\alpha  (x_k,t_k)$.

\section{Statement of the Main Comparison Results}\label{statements}

We begin with a result in the half-space case since it is, in fact, the main result.
\begin{theorem}\label{comp:HS} Assume that $\Omega$ is given by \eqref{Om-HS}, that either
    \hyp{Comp-1} or \hyp{Comp-2} holds. Then the \LCR holds for Problem~\eqref{Eqn}-\eqref{Vbc-HS},
    hence the \GCR also holds.  \end{theorem}

Because of the form of Assumption~\hyp{NC} or \hyp{NSE}, this result is twofold: indeed, the cases
of first-order equations and of second-order equations are rather different, even if their
proofs---given respectively in Sections~\ref{case1} and~\ref{case2}---contain common features.

As we pointed out above, Assumption~\hyp{Cont}---which is essential in \hyp{Comp-1} and
\hyp{Comp-2}---is nothing but the classical \eqref{cond:a31-p} in the first-order case and, in the
second-order one, since Theorem~\ref{comp:HS} deals with a flat boundary, we can drop the $Q$-term
in this assumption (or assume, equivalently, that it holds only for $Q\equiv0$).

The case of general domains is just a corollary of Theorem~\ref{comp:HS} because of
Proposition~\ref{lcr-scr}: indeed the fact that a \GCR reduces to a \LCR allows a local flattening
of the boundary, therefore to recover the half-space case.

We formulate anyway the result.  \begin{theorem}\label{comp:GC} Assume that \hyp{$\Omega$} holds,
that either \hyp{Comp-1} or \hyp{Comp-2} holds. Then the \LCR holds for
Problem~\eqref{Eqn}-\eqref{Vbc-GC} hence the \GCR also holds.  \end{theorem}

In the proofs of Theorem~\ref{comp:HS} or Theorem~\ref{comp:GC}, there is of course a
tremendous difference in showing that that the \LCR holds when:\\ $\hspace*{1em}\bullet$ either
$(x,t) \in \Omega \times (0,T)$ and we can choose $\bar r, \bar h>0$ such that
$\overline{Q_{x,t}^{\bar r,\bar h}} \subset \Omega \times (0,T)$\,;\\ $\hspace*{1em}\bullet$ or
$(x,t) \in \domeg \times (0,T)$\,.

While the first case is classical and requires arguments which are by now rather well-known---mostly
described in the User's guide of Crandall, Ishii and Lions~\cite{Users}---, new ideas are needed to
treat the latter one. Of course we only focus on this second case in the proofs; however, these new
ideas also have to be combined with ``classical comparison arguments", with which we assume that the
reader is familiar.

When, at several stages of the following comparison proofs, we refer to ``classical comparison
arguments", we mean those which can be found in \cite{Users}, \ie those with which one can obtain a
\LCR for $(x,t) \in \Omega \times (0,T)$.

\section{Proof of \LCR in the Half-Space Case in the First-Order Case}
\label{case1}

As the title of the section indicates it, we are going to consider
Problem~\eqref{Eqn}-\eqref{Vbc-HS} set in  $\Omega=\{(x',x_N)\in\R^{N}:x_N>0\}$. 

The aim of this section is to prove that a \LCR holds for any point $(\tilde x ,\tilde t)\in \Omegb
\times (0,T)$ and, of course, the only difficulty is when $\tilde x\in \domeg$, otherwise the result
just follows by a standard comparison argument if we choose $\bar r,\bar h$ small enough in order to
have $Q_{\tilde x ,\tilde t}^{\bar r,\bar h}\subset \Omega \times (0,T).$ Indeed, \hyp{Cont} is a stronger assumption than the
classical hypothesis under which such a comparison result holds.

For $\tilde x\in \domeg$, we are going to show that such a \LCR holds in $\overline{Q_{\tilde x
,\tilde t}^{\bar r,\bar h}}$ for any $\bar r>0$ and $0<\bar h<\tilde t$. To do so, we argue by
contradiction assuming that
\[ \max_{\overline{Q_{\tilde x ,\tilde t}^{\bar r,\bar h}}}(u-v)_+ > 
\max_{\partial_p {Q_{\tilde x ,\tilde t}^{\bar r,\bar h}} }(u-v)_+ .\]
{ Using \hyp{Gen}-$(i)$, we can assume without loss of generality that $u$ is a {\em
strict subsolution}, i.e. all its viscosity subsolution inequalities holds with $\leq -\eta <0$
instead of $\leq 0$; indeed, it suffices to replace $u(x,t)$ by
$$ \tilde u(x,t):=u(x,t)-\kappa(K t-x_N)\; ,$$
for $0<\kappa \ll 1$ and $K$ large enough to get the strict subsolution property. We show that the \LCR inequality
holds in $\overline{Q_{\tilde x,\tilde t}^{\bar r,\bar h}}$ for
$\tilde u$ and $v$, and then we let $\kappa$ tend to $0$. In the sequel, we keep the notation $u$ for the {\em strict} subsolution.}

Now, in $\overline{Q_{\tilde x ,\tilde t}^{\bar r,\bar h}}\times \overline{Q_{\tilde x ,\tilde t}^{\bar
r,\bar h}}$, we introduce the function 
\[ \Psi_{\e,L}(x,t,y,s):=u(x,t)-v(y,s)-\frac{|x'-y'|^2}{\e^2}-\frac{|t-s|^2}{\e^2}-L|x_N-y_N|, \]
where the parameters $\e>0$ and $L>0$ are going to be chosen small enough and large enough
respectively. 
 
This function achieves its maximum at $(\xb,\yb,\tb,\sb)$---we drop the dependence of this point in
$\e$ and $L$ in order to simplify the notations---and with a suitable choice of $\e$ and $L$ (small
enough and large enough respectively), we know that $(\xb,\tb),(\yb,\sb)\notin \partial_p Q_{\tilde x ,\tilde t}^{\bar r,\bar h}$ by our contradiction hypothesis since, by classical arguments,
\[ u(\xb,\tb)-v(\yb,\sb) \to \max_{\overline{Q_{\tilde x ,\tilde t}^{\bar r,\bar h}}}(u-v)_+
\quad \hbox{when  }\e\to 0, L\to +\infty.\]
Hence each point $(\xb,\tb),(\yb,\sb)$ belongs either to $Q_{\tilde x ,\tilde t}^{\bar r,\bar h}$ or
to $\overline{Q_{\tilde x,\tilde t}^{\bar r,\bar h}} \cap \{t=\tilde t\} $. But, according to
Proposition~\ref{regul}-$(i)$,
these two cases can be treated similarly. And a similar remark holds for all the maximum/minimum points we are going to encounter below.
 
Moreover, the classical estimates hold,
$$
\frac{|\bar x'-\bar y'|^2}{\e^2}+\frac{|\bar t-\bar s|^2}{\e^2}\to 0 \quad \hbox{as  }\e \to 0, L\to +\infty.
$$

\medskip

\noindent\textbf{(a)} \textit{We first prove that $\xb_N=\yb_N=0$ 
for a well-chosen constant $L$, large enough compared to $\eps$.}\smallskip

Indeed, let us start by assuming that $\xb_N\neq \yb_N$. We then face two situations:\\[2mm]
$(i)$ if $\xb_N>0$, whether $\xb_N-\yb_N$ is positive or negative we may use the inside equation
\begin{equation}\label{first.order.inside}
    a_\e + F(\xb,\tb,\pe \pm Le_N)\leq -\eta\,,
\end{equation}
where
\[ a_\e:=\frac{2(\tb-\sb)}{\e^2}\quad\hbox{and}\quad \pe:=\frac{2(\xb'-\yb')}{\e^2}.\]

\noindent $(ii)$ If $\xb_N=0$, then $|x_N-y_N|=-(x_N-y_N)$ if $x_N,y_N$ are close enough to $\xb_N,\yb_N$ and
the boundary condition yields
\begin{equation}\label{first.order.bndry}
     \min\left(a_\e + F\left(\xb,\tb,\pe - Le_N\right), L+G\left(\xb',\tb,\pe, \frac 2
     {\e^2}Id\right)\right)\leq -\eta\,.
\end{equation}
Now, it is clear that, for a choice of the form $L=C\e^{-2}$ with $C$ large enough, none of
\eqref{first.order.inside} or \eqref{first.order.bndry} can hold and therefore $\xb_N=\yb_N$.

\medskip

Next, we again argue by contradiction, assuming that $\xb_N=\yb_N>0$. As it is well-known,
we can add a term in the test-function in order that $(\xb,\yb,\tb,\sb)$ becomes a strict maximum
point. { We are not going to do it here in order to simplify matter, but we point out
that \hyp{Gen}-$(i)$ ensures that these additional terms would just produce small perturbations in
the inequalities.}

Then, regularizing the term $|x_N-y_N|$ by changing it into $(|x_N-y_N|^2+\alpha^2)^{1/2}$ for
$0<\alpha \ll 1$, at the new maximum point $(\xb_\alpha,\yb_\alpha,\tb_\alpha,\sb_\alpha)$, we have
in particular
\begin{equation}\label{ineq-pos1}
a_{\e,\alpha} + F\Big(\xb_\alpha,\tb_\alpha, p_{\e,\alpha} + 
L\frac{((\xb_\alpha)_N-(\yb_\alpha)_N)}{(|(\xb_\alpha)_N-(\yb_\alpha)_N|^2+\alpha^2)^{1/2}}\,e_N\Big) \leq -\eta ,
\end{equation}
where $a_{\e,\alpha} ,p_{\e,\alpha}$ are defined in the same way as $a_\e,\pe$ 
replacing $\xb,\yb,\tb,\sb$ by $\xb_\alpha,\yb_\alpha,\tb_\alpha,\sb_\alpha$. 
This inequality implies, using \hyp{NC}, that
\begin{equation}\label{key-est}
L\frac{((\xb_\alpha)_N-(\yb_\alpha)_N)}{(|(\xb_\alpha)_N-(\yb_\alpha)_N|^2+\alpha^2)^{1/2}}
=O(|p_{\e,\alpha}|+|a_{\e,\alpha}|)=o(\e^{-1}) ,
\end{equation}
this estimate being uniform \wrt $\alpha$. Notice that, in order to have the right estimate of
$a_{\e,\alpha}$, we need to double the variables in the same way for both $x'$ and $t$: this is where
the local Lipschitz continuity in $t$ of $F,G$ is required, \cf Remark~\ref{rem:Q.lip}.

Estimate~\eqref{key-est} is crucial since the easy estimate
$$ L\frac{((\xb_\alpha)_N-(\yb_\alpha)_N)}{(|(\xb_\alpha)_N-(\yb_\alpha)_N|^2+\alpha^2)^{1/2}}\leq L,$$
does not allow to carry out the classical arguments of the comparison proof for first-order Hamilton-Jacobi Equations since $L$ is of order $\e^{-2}$.

On the contrary,  with \eqref{key-est}, we easily obtain a contradiction for $\alpha$
small enough: indeed, the supersolution $v$ satisfies a similar inequality, just replacing
$\xb_\alpha$ by $\yb_\alpha$, $\tb_\alpha$ by $\sb_\alpha$ and $\leq -\eta$ by $\geq 0$, namely
\begin{equation}\label{ineq-pos2}
a_{\e,\alpha} + F\Big(\yb_\alpha,\sb_\alpha, p_{\e,\alpha} + 
L\frac{((\xb_\alpha)_N-(\yb_\alpha)_N)}{(|(\xb_\alpha)_N-(\yb_\alpha)_N|^2+\alpha^2)^{1/2}}\,e_N\Big) \geq 0.
\end{equation}
Hence, we
are in the same situation as in the classical proof with a doubling of variable and combining \eqref{ineq-pos1}-\eqref{ineq-pos2}
for $0<\alpha\ll 1$, using \eqref{key-est}, one obtains easily the desired contradiction for $\e$ small enough.
 
\medskip

The rest of the proof consists in dealing with the case $\xb_N=\yb_N=0$ since we have left out the
other cases.

\medskip

\noindent\textbf{(b)} \textit{We perform a twin blow-up \`a la Forcadel, Imbert and Monneau
\cite{FIM}.}\smallskip

To do so, we fix $\e$ and $L$ and for $0<\delta \ll 1$, we introduce the following functions
\begin{equation}\label{def.double.bup}
    \begin{aligned}
    u_\delta(x,t) &:= \frac{1}{\delta^2}\left(u(\xb'+\delta x', \delta^2  x_N,\tb +\delta  t)-
                    u(\xb,\tb)-\delta \pe \cdot x' -\delta a_\e t\right),\\
    v_\delta(y,s) &:= \frac{1}{\delta^2}
                    \left(v(\yb'+\delta y',\delta^2  y_N,\sb +\delta s)-v(\yb,\sb)-\delta\pe \cdot y' -\delta a_\e s\right).
\end{aligned}
\end{equation}
It is worth pointing out that, a priori, these functions need to be defined differently in the cases
when $\tb$ and/or $\sb$ can be equal to $\tilde t$ because we have to restrict ourselves to either $t\leq 0$ 
and/or $s\leq 0$. In order to unify all the cases, we always define these functions for $t,s\leq 0$; 
Proposition~\ref{regul} ensures that, as far as viscosity inequalities are concerned, the points on the upper 
boundary $t=0$ or $s=0$ behave as points in $\{t<0\}$ or $\{s<0\}$.

Hence $u_\delta$ is defined on the set
$$ \Theta_\delta:=\{(z,\tau): ((\delta z',\delta^2z_N),\delta \tau) \in \left[\overline{Q_{\tilde x,\tilde t}^{\bar r,\bar h}}-(\xb,\tb)\right],\tau \leq 0\},$$ 
while $v_\delta$ is defined on the set
$$ \Theta_\delta':=\{(z,\tau): ((\delta z',\delta^2z_N),\delta \tau) \in \left[\overline{Q_{\tilde x,\tilde t}^{\bar r,\bar h}}-(\yb,\sb)\right],\tau \leq 0\},$$
It is clear that both $\Theta_\delta, \Theta_\delta'$ tend to $\Omegb \times (-\infty, 0)$ when $\delta$ tends to $0$, where $\Omega$ is given by \eqref{Om-HS}.

Notice also that the Ventcell boundary condition forces us to use a different scaling in the tangent
variables ($x',t,y',s$) and in the normal ones ($x_N,y_N$) and to introduce the compensating terms
$a_\e$ and $p_\e$, two main differences with \cite{FIM}. 

Using the maximum property of function $\Psi_{\e,L}$ at $(\xb,\tb,\yb,\sb)$, we deduce the estimate 
\begin{equation}\label{ineq-ud-vd}
u_\delta(x,t)-v_\delta(y,s) 
\leq \frac{|x'-y'|^2}{\e^2}+\frac{| t - s|^2}{\e^2}+L| x_N- y_N| ,
\end{equation}
and $u_\delta(0,0)=v_\delta(0,0)$. 
This inequality shows that $u_\delta$ is locally bounded from above and $v_\delta$ is locally bounded from below, and 
both uniformly w.r.t. $\delta$.

%
%
%
The functions $u_\delta,v_\delta$ are a strict $\eta$-subsolution and a supersolution respectively of the following boundary problem 
\begin{equation}\label{eq.double.bup}
\begin{aligned}
    \delta w_t+ a_\e+ F(\xb' +\delta x', \delta^2  x_N, \tb +\delta  t , 
    \pe + \delta \Dxp  w+ w_{ x_N} e_N) = 0 &\text{ in  }\{ x_N>0\}\,,\\
     \qquad- w_{x_N}+G(\xb' +\delta x', \delta^2  x_N, \tb +\delta  t , 
    \pe + \delta \Dxp  w, \Ddxp   w) = 0 &\text{ on  }\{ x_N=0\}\,.
\end{aligned}
\end{equation}

\smallskip

\noindent\textbf{(c)} \textit{Reduction to locally bounded functions $u_\delta,v_\delta$.}
\smallskip

We are going to prove that we can assume w.l.o.g. that $u_\delta$ and $v_\delta$ are locally bounded.
For $u_\delta$, we use that, by \eqref{ineq-ud-vd}, $(0,0)$ is a maximum point of the function
\[(x,t)\mapsto u_\delta(x,t)- \frac{|x'|^2}{\e^2}-\frac{|t|^2}{\e^2}-L x_N\,.\]

By Proposition~\ref{prop:jets}, if $\underline{\lambda}$ is the maximal solution of
\begin{equation}\label{eq:ul}
a_\e + F(\xb', 0, \tb  , \pe + \lambda e_N)= -\eta\,,
\end{equation}
then for any $\lambda>\underline{\lambda}$, $(p_x, p_t,X')=((0,\lambda),0,  2 \e^{-2}\Id)\in
\PP^{2,+}_r u_\delta(0,0)$. Letting $\lambda \to \underline{\lambda}$, we have at the same time
$a_\e + F(\xb', 0, \tb  , \pe +  \underline{\lambda} e_N)= -\eta $ and 
\[  
    \min\Big(a_\e + F(\xb', 0,\tb  , \pe +  \lambda e_N), - \lambda 
    +G(\xb', 0, \tb, \pe , 2\e^{-2}\Id)\Big)\leq -\eta\,,
\]
for any $\lambda \in [\underline{\lambda},+\infty)$.

We deduce from these inequalities that, for any $K>0$, the functions
\[\underline \psi_K(x,t):=-K+\underline{\lambda} x_N+ \frac{|x'|^2}{\e^2}+\frac{|t|^2}
{\e^2} ,\]
are
approximate $\eta$-strict subsolutions of the problem in a neighborhood of $(0,0)$. Meaning, it is a subsolution
where $\leq -\eta $ is replaced by $\leq -\eta + o_\delta (1)$ to take care of the terms like $\delta x'$,
$\delta  t$,  $\delta w_t$, $\delta \Dxp  w$ in the equations.

The consequence is that $\tilde u_\delta:=\max(\underline \psi_K ,u_\delta)$ is an approximate
$\eta$-strict subsolution as well,
which is bounded from below. 

A much easier but similar argument allows to bound $v_\delta$ from above by just using the 
coercivity of $F$ in $p$, \cf \hyp{NC}, and in particular in the $p_N$-direction, by building a supersolution $\overline \psi_K$ of the form $\overline \psi_K (x,t):=K +\bar \lambda x_N$ for $\bar \lambda \ll -1$ and $K>0$ large. We denote by $\tilde v_\delta=\min(\overline \psi_K, v_\delta)$ the supersolution that we obtain in that way.

Finally choosing $K$ large enough---depending on $\e$---, we can assume \wlg that $\tilde u_\delta,\tilde v_\delta$ still satisfy
Inequality~\eqref{ineq-ud-vd}, at least in a neighborhood of $(0,0,0,0)$.

\medskip

\noindent\textbf{(d)} \textit{Passage to the limit in both viscosity inequalities and in the maximum
point property \eqref{ineq-ud-vd}.}\smallskip

Since the functions $\tilde u_\delta, \tilde v_\delta$ are uniformly bounded \wrt $\delta$, there is
no difficulty in the passages to the limit in the viscosity inequalities by using the half-relaxed
limit method. But the problem is that it is not clear---and perhaps false---that $\tilde u=\limssup \tilde
u_\delta$ and $\tilde v=\limiinf \tilde v_\delta$ still satisfy \eqref{ineq-ud-vd}.

To overcome this difficulty, we use the ``sup-convolution trick'' of Forcadel, Imbert and Monneau \cite{FIM}\footnote{The authors are indebted
to Nicolas Forcadel for explaining them in full details this trick, pointing out in particular all
its advantages.}. To do so, for $0<\beta \ll \e^2$,
we introduce the function
\begin{equation}\label{SC0}
 \tilde u_\delta^\beta (x,t):= \sup_{(z,\tau)\in \Theta_\delta}
\left(\tilde u_\delta (z,\tau)-\frac{|x'-z'|^2}{\beta^2}-\frac{|t-\tau |^2}{\beta^2}-\frac 1
    \beta |x_N-z_N| \right),
\end{equation}
where we recall that $\Theta_\delta$ is the set where $\tilde u_\delta$ is defined (see above).

Since the functions $\tilde u_\delta$ are uniformly bounded, usual arguments imply that, for fixed
$\beta$, the functions are locally uniformly bounded
and Lipschitz continuous \wrt all variables. Moreover, still for fixed $\beta$, we have for any $\delta$, on one hand
\begin{equation}\label{keyineqSC1}  \tilde u_\delta \leq  \tilde u_\delta^\beta,
\end{equation}
and, on the other hand (see in the Appendix, Section~\ref{app:supconv} for the computation), if $\beta^{-1}\geq L$
\begin{equation}\label{keyineqSC2}
\tilde u^\beta_\delta (x,t)-\tilde v_\delta (y,s) \leq \Ceb\left(\frac{|x'-y'|^2}{\e^2}+\frac{|t-s|^2}{\e^2}\right)+
L |x_N- y_N|,
\end{equation}
where $\Ceb:= \left(1-\frac{\beta^2}{\e^2}\right)^{-1}$.

Next we use Ascoli's Theorem: since the $\tilde u_\delta^\beta$ are uniformly bounded and Lipschitz continuous---at least locally---, there exists a subsequence
denoted by $\tilde u_{\delta'}^\beta$ which converges locally uniformly to a continuous function
$U^\beta$. And we set
$\tilde u=\limssup \tilde u_{\delta'}$ and $\tilde v=\limiinf \tilde v_{\delta'}$, \ie we take the
half-relaxed limits along this subsequence. Taking the $\limssup$ in \eqref{keyineqSC1} and
\eqref{keyineqSC2}, we obtain $ \tilde u \leq U^\beta$ and
$$U^\beta (x,t)-\tilde v (y,s) \leq \Ceb\left(\frac{|x'-y'|^2}{\e^2}+\frac{|t-s|^2}{\e^2}\right)+
L |x_N- y_N|.$$
As a consequence, we have, for any $\beta$ small enough
$$\tilde u (x,t)-\tilde v (y,s) \leq \Ceb\left(\frac{|x'-y'|^2}{\e^2}+\frac{|t-s|^2}{\e^2}\right)+
L |x_N- y_N|\,,$$
and letting $\beta$ tend to $0$, we recover \eqref{ineq-ud-vd}, \ie
\begin{equation}\label{keyineq}
\tilde u (x,t)-\tilde v (y,s) \leq \frac{|x'-y'|^2}{\e^2}+\frac{|t-s|^2}{\e^2}+
L |x_N- y_N|.
\end{equation}
We notice that this inequality implies $\tilde u (0,0)-\tilde v (0,0)\leq 0$ and therefore $\tilde u (0,0)=\tilde v (0,0)= 0$ because, by definition
$\tilde u (0,0)\geq 0$ and $\tilde v (0,0)\leq 0$.

As we mention it above, there is no difficulty in passing to the limit in the viscosity inequalities
and, by standard stability results, the following properties hold true:
if $ x_N, y_N>0$ we have
\[ a_\e+F(\xb,\tb, \pe+\tilde u_{x_N} e_N) \leq -\eta <0\leq a_\e+F(\yb,\sb, \pe +\tilde v_{x_N} e_N) ,\]
while, on the boundary, we have
\[\begin{aligned} 
    \min(a_\e+F(\xb,\tb, \pe +\tilde u_{x_N} e_N), -\tilde u_{x_N}+G(\xb',\tb, \pe ,\Ddxp   \tilde u))
    &\leq -\eta  ,\\
    \max(a_\e+F(\yb,\sb, \pe +\bar v_{x_N} e_N), -\bar v_{x_N}+G(\yb',\sb, \pe ,\Ddxp   \bar v))
    &\geq 0 . 
\end{aligned}\]
Moreover we have
$$\tilde u=\max(\underline \psi_K ,\bar u)\; , \; \tilde v=\min(\overline \psi_K ,\bar v),$$
where $\bar u=\limssup u_{\delta'}$ and $\bar v=\limiinf v_{\delta'}$.

The next step consists in regularizing $\tilde u$ and $\tilde v$ by tangential sup and
inf-convolution respectively. That is, as we did above but only in $x'$ and $t$. To do so, we
introduce the following notations: if $\chi$ is an \usc (resp. \lsc) function defined on $\Omegb
\times \R$ we set, for $0<\alpha \ll 1$,
\begin{equation}\label{sup-conv}
 \chi^\alpha (x,t):= \sup_{(z',\tau)\in\R^{N-1}\times\R}
\left(\chi((z',x_N),\tau)-\frac{|x'-z'|^2}{\alpha^2}-\frac{|t-\tau |^2}{\alpha^2}\right) ,
\end{equation}
(resp.
\begin{equation}\label{inf-conv}
\chi_\alpha (y,s):= \inf_{(z',\tau)\in\R^{N-1}\times\R}
\left(\chi((z',y_N),\tau)+\frac{|y'-\tilde z'|^2}{\alpha^2}+\frac{|s- \tau |^2}{\alpha^2}\right)
 .)
 \end{equation}
These functions are well-defined as soon as $\chi$ has a subquadratic behavior in $(x',t)$ either from above or from below,
which is the case for all the
functions $\bar u, \bar v, \underline \psi_K ,\overline \psi_K$. And with these notations, we have, for $0<\alpha \ll \e$
$$\tilde u^\alpha =\max((\underline \psi_K)^\alpha  ,\bar u^\alpha)\; , \; \tilde v_\alpha=\min((\overline \psi_K)_\alpha ,\bar v_\alpha).$$

Applying the sup-inf convolution to Inequality~\eqref{keyineq}, using a similar computation as in the
Appendix, Section~\ref{app:supconv}, gives
\begin{equation}\label{keyineqbis1}
\tilde u^\alpha (x,t)-\tilde v_\alpha (y,s) \leq \Cea\left(\frac{|x'-y'|^2}{\e^2}+\frac{|t-s|^2}{\e^2}\right)+
L |x_N- y_N|,
\end{equation}
where $\Cea:= \left(1-\frac{\alpha^2}{\e^2}\right)^{-1}$.

Hence $(0,0,0,0)$ is still a maximum point of the function
$$
\tilde u^\alpha (x,t)-\tilde v_\alpha (y,s)- \Cea\left(\frac{|x'-y'|^2}{\e^2}+\frac{|t-s|^2}{\e^2}\right)-
L|x_N-
y_N|.
$$

Now we claim that $\tilde u^\alpha, \tilde v_\alpha$ satisfy exactly the same viscosity inequalities as $\tilde u, \tilde v$
respectively: proving these properties do not present any difficulty since the nonlinearities involved in the limiting problem
do not depend neither on $x$ nor on $t$; the twin blow-up has the effect to ``freeze'' the dependence in $x$ and $t$.

We also recall that the functions $\tilde u^\alpha$ and $ \tilde v_\alpha$ are Lipschitz continuous in the tangent variables $(x',t)$ 
uniformly \wrt $x_N$ or $y_N$, $\tilde u^\alpha$ being semi-convex and $\tilde v_\alpha$ semi-concave.  And because of \hyp{NC},
$\tilde u^\alpha$ is also Lipschitz continuous in $x_N$ uniformly \wrt the tangent variables $(x',t)$ but only for $x_N>0$; at this point, 
we may still have a discontinuity at $x_N=0$.

To solve this issue, we use Proposition~\ref{regul} together with this Lipschitz continuity in $x_N$: this implies that $\tilde u^\alpha$ is 
necessarily continuous at any point $((x',0),t)$ in a neighborhood of $(0,0)$. Hence $\tilde u^\alpha$
is Lipschitz continuous up to the boundary $\{x_N=0\}$ \wrt all variables. And concerning $\bar v_\alpha$, it may still present 
discontinuities but, again thanks to Proposition~\ref{regul}, it satisfies
$$  \tilde v_\alpha ((x',0),t) = \liminf_{{\substack{ (y,s)\to ((x',0),t)\\
    (y,s)\in \Omega \times \R}}} \tilde v_\alpha (y,s),$$
in a neighborhood of $(0,0)$.

The main consequence of these properties are that $\tilde u^\alpha = \bar u^\alpha$ in a neighborhood of $(0,0)$ because the
continuous function $(\underline \psi_K)^\alpha$ satisfies the inequality
$(\underline \psi_K)^\alpha(0,0)<0$. 


For $\tilde v_\alpha$, the situation is more complicated since this function may present discontinuities in $x_N$ but we also have
$\tilde v_\alpha(y_k,s_k)= \bar  v_\alpha(y_k,s_k)$ for any sequence $(y_k,s_k)\to ((x',0),t)$ such that 
$$  \tilde v_\alpha(y_k,s_k)\to \tilde v_\alpha ((x',0),t)\; .$$
Roughly speaking, $\tilde v_\alpha=\bar  v_\alpha$ for all the points which plays a role in our
arguments. So, in the sequel, we
are going to argue with $\bar u^\alpha$ and $\bar  v_\alpha$.

\medskip

\noindent\textbf{(e)} \textit{Reduction to a strict subsolution which is smooth \wrt the $x'$-variable.}
\smallskip

As we mentioned above, passing to the limit in the equations is easy and we get that
$\bar u^\alpha$ is a subsolution of the following Ventcell problem
\begin{equation}\label{eq:Ventcell-bup}
    \begin{cases}
        \tilde F(u_{x_N}) \leq 0\,,\\ 
        -u_{x_N}+\tilde G(\Ddxp  u)\leq0\,,
    \end{cases}
\end{equation}
where, for any $p_N\in \R$ and any $(N-1)\times (N-1)$ symmetric matrix $M'$,
    \[ 
    \tilde F(p_N):= a_\e+F(\xb,\tb, \pe +p_N e_N)+\eta\,,\quad \tilde G(M'):=G(\xb',\tb, \pe
    ,M')+\eta\,.
    \]
Notice that $\tilde F(p_N), \tilde G(M')$ are continuous, quasiconvex functions (because of \hyp{NC}-\hyp{QC-G}), and $\tilde G$ satisfies the
ellipticity condition.

Moreover $(0,0)$ is a maximum point of 
$$
\bar u^\alpha (x,t)-\bar  v_\alpha (y,s)- \Cea\left(\frac{|x'-y'|^2}{\e^2}+\frac{|t-s|^2}{\e^2}\right)-
L|x_N-
y_N|.
$$
and changing $\bar u^\alpha(x,t)$ in $\bar u^\alpha(x,t)-|x|^2-t^2$, we may assume \wlg that $(0,0)$ is a {\em strict} maximum point of this function.
In particular, by restricting ourselves to $x_N=y_N$, $(0,0)$ is a {\em strict} maximum point of 
\begin{equation}\label{keyineqter1}
    \begin{aligned}
        \chi(x',y',t,s,x_N):=\bar u^\alpha ((x',x_N),t) &- \bar  v_\alpha ((y',x_N),s)\\ 
        &- \Cea\left(\frac{|x'-y'|^2}{\e^2}+\frac{|t-s|^2}{\e^2}\right).
    \end{aligned}
\end{equation}
In addition, this new $\bar u^\alpha(x,t)$ is still a subsolution of a Ventcell problem like
\eqref{eq:Ventcell-bup} in a neighborhood of $(0,0)$---say for $t\in (-r,0], x'\in B(0,r), x_N\in
[0,r)$ for some $r>0$ small enough---, changing perhaps $\eta$ in $\eta/2$.

Since $F,G$ satisfy Assumptions~\hyp{Comp-1}, we can apply Corollary~\ref{regul-smooth} in the
Appendix, Section~\ref{app:regsubsol}\footnote{Compared to the appendix, we have here the additional variable $t$ but, since it
acts only as a parameter, this does not create any difficulty.}. In order to simplify the notations, we
are not going to introduce new parameters or points but this result allows us to reduce to the case
when the (strict) subsolution is smooth \wrt the $x'$-variable and has a strict maximum point at
$(0,0)$\footnote{In fact the new maximum point is not necessarily at $(0,0)$ anymore but by the {\em
strict} maximum point property, it has moved at a nearby point and we may also assume that this new
point is also a {\em strict}
maximum point.}.
\begin{remark}\label{simple-pb}
We point out that one of the main effect of the twin blow-up is to reduce to a rather simple Ventcell problem like \eqref{eq:Ventcell-bup}
where only $u_{x_N}$ appears in $\tilde F$ and $\tilde G$ only depends on $\Ddxp  u$; in this simplified context, we can apply without
any difficulty the results of the Appendix, Section~\ref{app:regsubsol}---in particular Corollary~\ref{regul-smooth}. This allows us to use
rather weak quasiconvexity assumptions on $F$ and $G$.
\end{remark}

\newpage

\noindent\textbf{(f)} \textit{To the conclusion}\smallskip

In order to obtain the conclusion, we are going to argue in a quasi-classical way, by doubling the variables:
we set 
$$ \gamma:= G\left(\xb', \bar t, p_\e ,\Ddxp  \bar u^\alpha (0,0)\right),$$
and we introduce the function $\Psi(x,t,y,s)$ which is given by
$$
\bar u^\alpha(x,t)-\bar v_\alpha(y,s)- [\gamma+\eta/2] (x_N-y_N)
-\frac{(x_N-y_N)^2}{\beta^2}- \Cea\left(\frac{|x'-y'|^2}{\e^2}+\frac{|t-s|^2}{\e^2} \right),$$
where $\beta>0$ is a new parameter devoted to tend to $0$. We recall that $\alpha\ll\e$.

We look at maximum points of $\Psi $ in the domain $x',y' \in \overline{B(0,r')}, x_N,y_N  \in [0,r']$ and $t,s\in [-r',0]$ where $r'<r$ is
chosen in such a way that $(0,0)$ is the only maximum point of the above function $\chi$ in this domain.


From now on, dropping the dependence of the maximum points in $\e,\alpha,\beta$, we just denote by $(x,t,y,s)$ a maximum point of $\Psi$. By the strict 
maximum point property, we know that $(x,t,y,s)\to (0,0,0,0)$ when $\beta \to 0$. Therefore, if $\beta$ is small enough, $(x,t),(y,s)\in B(0,r')\times [0,r')\times (-r',0] $ and we can look at the viscosity sub and supersolution inequalities.

If $x_N>0$ and $y_N>0$, the $F$-inequalities are both satisfied and we immediately obtain a
contradiction by classical arguments. It is worth pointing out anyway that we have to deal at the
same time with the parameters $\beta$ and $\e$ (and $\alpha$) but the coercivity of $F$, which we have already used
many times, allow to first let $\beta$ tend to $0$ and then $\e$ since all the gradients which are involved remain bounded.

Next, we have to examine the case when $x_N=0$ or $y_N=0$.

We first want to apply the Jensen-Ishii Lemma and to do so, we consider the parabolic semijets introduced
in Section~\ref{sec:jets}. We keep here only the parts of the reduced parabolic sub
and superjets $\PP^{2,-}_r, \PP^{2,+}_r$ which play a role in the equation, \ie the couples
corresponding to $(D_{x_N},\Ddxp  )$. Moreover, in order to use the Jensen-Ishii Lemma, we use
the closures $\bar\PP^{2,+}_r,\bar\PP^{2,-}_r$, \cf the end of Section~\ref{sec:jets}.

We deduce that there exist $(N-1)\times (N-1)$ symmetric matrices $X',Y'$ such that
\begin{equation}\label{ineqmat2}
\left[\begin{array}{cc}X'&0\\0&-Y'\end{array}\right]
 \le 2\Cea \left[\begin{array}{cc}I&-I\\-I&I\end{array}\right];
\end{equation}
and $\lambda_1,\lambda_2 \in \R$ such that
\begin{enumerate}
    \item[$(i)$] $(\lambda_1,X')\in \bar \PP^{2,+}_r \bar u^\alpha  (x,t)$,
    \item[$(ii)$] $(\lambda_2,Y')\in \bar \PP^{2,-}_r \bar v_\alpha (y,s)$,
    \item[$(iii)$]  in addition, we have $X'= \Ddxp   \bar u^\alpha(x,t)$ (because of the
        regularity of $u^\alpha$) and
$$ \lambda_1= \lambda_2= \gamma-\eta/2+ \frac{2(x_N-y_N)}{\beta^2}\,.$$
\end{enumerate}

We first examine the case when $x_N=0$: if the Ventcell boundary condition holds, we would have
$$ -\lambda_1+ G(\xb', \bar t, p_\e , X')\leq -\eta,$$
but, taking into account the definition of $\gamma$, this would mean
$$\left[G(\xb', \bar t, p_\e ,\Ddxp   \bar u^\alpha(x,t))-G(\xb', \bar t, p_\e ,\Ddxp   \bar u^\alpha(0,0))\right] + \eta/2 + \frac{2y_N}{\beta^2} \leq -\eta.$$
Clearly this inequality cannot hold if $\beta$ is small enough since $(x,t) \to (0,0)$ when $\beta\to 0$ and the continuity of $\Ddxp   \bar u^\alpha$
implies that $\Ddxp   \bar u^\alpha(x,t)\to \Ddxp   \bar u^\alpha(0,0)$. Hence either $x_N>0$ or the $F$-inequality holds for $\bar u^\alpha$.

On the other hand, if $y_N=0$ and if the Ventcell boundary condition holds, we would have
$$ -\lambda_2+ G(\yb, \sb, p_\e , Y')\geq 0.$$
Here the argument is slightly more complicated since the equation is taken at the point $(\yb', \bar s)$. We use \hyp{cont} for $G$
which implies
$$ G(\yb',\sb, p_\e , Y')-G(\xb', \bar t, p_\e ,X')\leq $$
$$\omega\left((|\xb'-\yb'|+|\tb-\sb])|p_\e|+
\Cea\left(\frac{|\xb'-\yb'|^2}{\e^2}+\frac{|\tb-\sb|^2}{\e^2} \right)\right)=\theta(\e),$$
and the right-hand side tends to $0$ when $\e$ tends to $0$. Therefore
\begin{align*}
-\lambda_2+ G(\yb, \sb, p_\e , Y')=& G(\yb, \sb, p_\e , Y') -G(\xb', \bar t, p_\e ,\Ddxp  \bar u^\alpha
(0,0))+ \eta/2- \frac{2x_N}{\beta^2}  \\
\leq \theta(\e) + \eta/2- \frac{2x_N}{\beta^2},
\end{align*}
and if $\beta, \e$ are small enough, the quantity in the right-hand side is strictly negative. So,
for $\bar v_\alpha$, the Ventcell condition cannot hold neither and therefore we have the
$F$-inequality also for $\bar v_\alpha$.

Having the $F$-inequalities for both $\bar u^\alpha$ and $\bar v_\alpha$, we conclude easily by
letting first $\beta \to 0$, using the normal coercivity to control the $x_N$-derivatives. 
Then we let $\e\to0$ (and of course, $\alpha=o(\e)$ too), and we get 
a contradiction by standard arguments using $F$. The proof is then complete.  

\section{Proof of \LCR in the Half-Space Case in the Second-Order Case}
\label{case2}

For second-order equations, the strategy is exactly the same and we are not going to repeat all the details here.
{ Again $u, v$ denotes an \usc strict subsolution and a \lsc supersolution of \eqref{Eqn}-\eqref{Vbc-HS} in
$Q_{\tilde x ,\tilde t}^{\bar r,\bar h}$ respectively.}

But the first step has to be done differently since, in the first-order case, we reduce to the case when the
maximum point satisfies $x_N=y_N=0$ by a combination of normal coercivity and use of Ventcell boundary 
condition. Here, on the contrary, we only use the normal ellipticity of $F$.

We start by assuming that 
\begin{equation}\label{max-bord}
M:=\max_{
    \overline{Q_{\tilde x ,\tilde t}^{\bar r,\bar h}}}(u-v)_+ > 
\max_{\partial_p {Q_{\tilde x ,\tilde t}^{\bar r,\bar h}} }(u-v)_+
\end{equation}
and we denote by $(\xb,\tb)\in Q_{\tilde x ,\tilde t}^{\bar r,\bar h}$ a point where $M>0$ is
attained.

\medskip

\noindent\textbf{(a)} \textit{Forcing the maximum to be attained at the boundary.}\\[2mm]
For $\tau\in \R$, we set
\[ \varphi(\tau):=\tau- \frac{\tau^2}{2} ,\]
and, in $\mathcal{E}_\e:=\overline{Q_{\tilde x ,\tilde t}^{\bar r,\bar h}}\times \overline{Q_{\tilde x ,\tilde t}^{\bar
r,\bar h}} \cap \{|x_N-y_N|\leq \e\}$, we introduce the function 
\[ \Psi_{\e,L}(x,t,y,s):=u(x,t)-v(y,s)-\frac{|x'-y'|^2}{\e^2}-\frac{|t-s|^2}{\e^2}-L\varphi \left(\frac{|x_N-y_N|}{\e}\right) ,
\]
where the parameters $\e>0$ and $L>0$ are going to be chosen small enough and large enough
respectively. We denote by $(\xb,\tb,\yb,\sb)$ a point of maximum of $\Psi_{\e,L}$ in
$\mathcal{E}_\e$, dropping the dependence in $\e$ and $L$ for simplicity of notations.

Notice that this penalization procedure is not as standard as usual
and the following result replaces Step~\textbf{(a)} from the first-order case

\begin{lemma}\label{lem:pen.eps}
    For $\e>0$ small enough and $L>0$ large enough (but independent of $\e$), the maximum point $(\xb,\tb,\yb,\sb)$ satisfies
    $\bar x_N=\bar y_N=0$.
\end{lemma}

\begin{proof}
    We proceed in three steps as follows.

    \noindent \textbf{1.}
    Notice first if $L$ is chosen large enough---with a size depending only on $u$ and $v$---, 
    the maximum of function $\Psi_{\e,L}$ cannot be achieved for $|x_N-y_N|= \e$. 
    More precisely, take $L$ such that
    \[ \frac{L}{2} \geq \max_{\overline{Q_{\tilde x ,\tilde t}^{\bar r,\bar h}}} u 
    -  \min_{\overline{Q_{\tilde y ,\tilde s}^{\bar r,\bar h}}} v\,. \]
    If $|\bar x_N-\bar y_N|=\eps$, the value of the $L\varphi$-term is $L/2$, implying the
    inequality  $\Psi_{\e,L}(\bar x,\bar t,\bar y,\bar s)\leq0$. But on the other hand,
    $\max(\Psi_{\e,L})\geq\Psi_{\e,L}(\bar x,\bar x,\bar t,\bar t)=M>0$, which yields a
    contradiction.

    Moreover, \eqref{max-bord} implies that $(\xb,\tb,\yb,\sb)$
    is necessarily in $Q_{\tilde x ,\tilde t}^{\bar r,\bar h}\times Q_{\tilde x ,\tilde t}^{\bar r,\bar h}$ for $\e$ small enough.
   
    \noindent \textbf{2.}
    A second remark is that, if $\bar x_N\neq \bar y_N$, the $\varphi$-term becomes smooth at these
    points. Hence, for instance if $\bar x_N>0$, we can use 
    \[(x,t)\mapsto v(\bar y,\bar s)+\frac{|x'-\bar y'|^2}{\e^2}+\frac{|t-\bar s|^2}{\e^2}+L\varphi
    \left(\frac{|x_N-\bar y_N|}{\e}\right)\]
    as a test-function in the inside equation for $u$ at $(\xb,\tb)$, which yields
    \[a_\e+F\Big(\xb,\tb,p_\e\pm\frac{L}{\e}e_N,\frac{2}{\e^2}\Id-\frac{L}{\e^2}e_N\otimes e_N\Big)\leq0\,.\] 
    But this contradicts the ellipticity of $F$ for $L$ large enough, its size depending only on the properties of $F$. Similarly, we also reach a
    contradiction if $\bar y_N>0$ by using the supersolution inequality for $v$, involving the $+L\e^{-2}\,
    e_N\otimes e_N$ term in $F$.

    \noindent \textbf{3.} At this stage, we are left with proving that $\bar x_N=\bar y_N>0$ cannot
    occur which is not as simple as in the first-order case. 
    We first notice that, by usual arguments, we can assume w.l.o.g that $(\xb,\tb,\yb,\sb)$ is a strict
    maximum point by subtracting $|x-\bar x|^4+|y-\bar y|^4+|t-\bar t|^4+|s-\bar s|^4$ to the
    function $\Psi_{\e,L}$---we keep the same notation for this new function. 

    Then, we denote by $\Psi_{\e,L,\alpha}$ the function which is the same as $\Psi_{\e,L}$ except
    that we replace $\e$ by $\alpha$ in the $\varphi$-term, more precisely
    \[ \Psi_{\e,L,\alpha}(x,t,y,s):=u(x,t)-v(y,s)-\frac{|x'-y'|^2}{\e^2}-\frac{|t-s|^2}{\e^2}-L\varphi
    \left(\frac{|x_N-y_N|}{\alpha}\right). \]
Notice however that we keep considering the maximum of $\Psi_{\e,L,\alpha}$ in $\mathcal{E}_\e$\
    i.e. the domain remains independent of $\alpha$.
    
    We first remark that, for $\alpha >\e$, as long as the maximum point of $\Psi_{\e,L,\alpha}$ in
    $\mathcal{E}_\e$ satisfies $x_N=y_N$, then this point is necessarily $(\bar x,\bar y,\bar t,\bar
    s)$. Indeed, this derives from the fact that $\Psi_{\e,L,\alpha}(x,t,y,s)=\Psi_{\e,L}(x,t,y,s)$ if $x_N=y_N$ and
$(\bar x,\bar y,\bar t,\bar s)$ is the only maximum point of $\Psi_{\e,L}(x,t,y,s)$ with the constraint $x_N=y_N$.

    Next, we define $\bar \alpha$ as the supremum of all $\alpha\geq \e$ such that the maximum of
    $\Psi_{\e,L,\alpha}$ in $\mathcal{E}_\e$ is still achieved for $x_N= y_N>0$, $i.e.$ for which
    $(\bar x,\bar y,\bar t,\bar s)$ is still a maximum point. We face several cases:
    \begin{enumerate}
\item[$(i)$] If $\bar \alpha = +\infty$, we can drop the $\varphi$-term: $(\xb,\tb,\yb,\sb)$
is a maximum point of the function 
$$(x,t,y,s)\mapsto u(x,t)-v(y,s)-\frac{|x'-y'|^2}{\e^2}-\frac{|t-s|^2}{\e^2},$$
or
$$(x,t,y,s)\mapsto u(x,t)-v(y,s)-\frac{|x-y|^2}{\e^2}-\frac{|t-s|^2}{\e^2},$$
and the classical comparison arguments given in the User's guide \cite{Users} can be performed,
remarking that \hyp{Cont} is even stronger than the structure conditions needed in \cite{Users}.
And they lead to a contradiction. 
    \item[$(ii)$] If $\bar \alpha <+\infty$ we distinguish two sub-cases:
    \begin{enumerate}
        \item[$(ii)$-$(a)$] If $(\bar x,\bar y,\bar t,\bar s)$ is a {\em strict} maximum point of
            $\Psi_{\e,L,\bar \alpha}$, then, for any $\alpha > \bar \alpha$, there is a maximum point
            $(\bar x_\alpha,\bar y_\alpha,\bar t_\alpha,\bar s_\alpha)$ such that $(x_\alpha)_N\neq
            (y_\alpha)_N$ and the sequence $(\bar x_\alpha,\bar y_\alpha,\bar t_\alpha,\bar
            s_\alpha)$ converges to the strict maximum point $(\bar x,\bar y,\bar t,\bar s)$ as
            $\alpha \to \bar \alpha$. In this case, we can apply the comparison arguments of \cite{Users}
            for the points $(\bar x_\alpha,\bar y_\alpha,\bar t_\alpha,\bar
            s_\alpha)$ since the $\varphi$-term is smooth if $x_N\neq y_N$, see Step~2. above.
            Taking $\alpha$ close enough to $\bar \alpha$, we can conclude for $\e$ small enough
            since 
            $$\frac{|\bar x_\alpha'-\bar y_\alpha'|^2}{\e^2}+\frac{|\bar t_\alpha-\bar s_\alpha|^2}{\e^2}=o_\e(1) + o_\alpha(1) \quad \hbox{and}\quad (\bar x_\alpha)_N-(\bar y_\alpha)_N=o_\alpha(1) 
            .$$
       
        \item[$(ii)$-$(b)$] If $(\bar x,\bar y,\bar t,\bar s)$ is NOT a {\em strict} maximum point of
            $\Psi_{\e,L,\bar \alpha}$, this means that there exists a sequence $(\bar x_k,\bar
            y_k,\bar t_k,\bar s_k)$ of maximum points of $\Psi_{\e,L,\bar \alpha}$ which converges
            to $(\bar x_\balpha,\bar y_\balpha,\bar t_\balpha,\bar s_\balpha)$ and such that
            $(x_k)_N\neq (y_k)_N$. Indeed, we cannot have $(x_k)_N=(y_k)_N$ (as $k\to\infty$)
            since $(\bar x_\balpha,\bar y_\balpha,\bar t_\balpha,\bar s_\balpha)$ is a {\em strict}
            maximum point of $\Psi_{\e,L,\bar \alpha}=\Psi_{\e,L}$ with the constraint $x_N=y_N$.
            And we conclude as in the previous case, by using the comparison arguments for the maximum point
            $(\xb_k,\tb_k,\yb_k,\sb_k)$.
    \end{enumerate}
\end{enumerate}

%
In any case, we reach a contradiction when the maximum point $(\xb,\tb,\yb,\sb)$ satisfies $\bar
x_N=\bar y_N>0$, so that we can assume w.l.o.g. that $\Psi_{\e,L}$ has a maximum point such that 
$x_N=y_N=0$.
\end{proof}

\begin{remark}
    In the proof of Lemma~\ref{lem:pen.eps}, even if this may not be completely crucial, we benefit
    from the same doubling of variables in $x'$ and $t$ since it simplifies matter, at least. This
    is where the local Lipschitz continuity in $t$ plays a role, \cf Remark~\ref{rem:Q.lip}.
\end{remark}

\medskip

\noindent\textbf{(b)} \textit{The twin blow-up argument.}\\[2mm]
After this first step, we perform the twin blow-up argument as in the first-order case, see
\eqref{def.double.bup}. Of course, since $F$ now depends on the second-derivatives, the equation
inside the domain involves more terms than in \eqref{eq.double.bup}, but we are not going to write
them here since passing to the limit yields a simple formulation in the end---see below. 

In order to reduce to the case when $u_\delta$ and $v_\delta$ are bounded as in the first-order case, we use sub- and
supersolutions of the form
\[ \psi^\pm(x,t):=\pm K_1(1-x_N-K_2x_N^2) ,\]
$\psi^-$ being the subsolution and $\psi^+$ the supersolution. The $K_1$-constant is used to take care of the Ventcell
boundary condition, while the $K_2$-one is used for the equation, using the ellipticity of $F$ in the normal direction. Both constants depend on $\e$ (but not on $\delta$) and we consider these sub and supersolutions only in a small
neighborhood of the boundary, \ie for $x_N$ small.

As in the first-order case, we define $u_\delta, v_\delta$ and
$$ \tilde u_\delta:= \max(u_\delta,\psi^-) \quad \hbox{and}\quad \tilde v_\delta:= \min(v_\delta,\psi^+)\; ,$$ 
which are now bounded sub and supersolutions. 

At this stage, we use the ``sup-convolution trick'' of Forcadel, Imbert and Monneau \cite{FIM} which allows us to pass to the limit through the
half-relaxed limits method both in the maximum point property and in the viscosity solutions inequalities using that $\varphi(s) \leq s$ for any $s$: 
taking $\limssup$ and $\limiinf$ along 
subsequences, we set $\tilde u=\limssup \tilde u_\delta$, $\bar u=\limssup u_\delta$, $\tilde v=\limiinf \tilde v_\delta$, $\bar v=\limiinf v_\delta$ and 
we have
$$ \tilde u:= \max(\bar u,\psi^-) \quad \hbox{and}\quad \tilde v_\delta:= \min(\bar v,\psi^+)\; .$$

\smallskip

\noindent\textbf{(c)} \textit{The limit problem.}\\[2mm]
Using \hyp{NSE} for the equation associated to $u_\delta, v_\delta$, it is clear enough that the leading terms are
the $\delta^{-2}(u_\delta)_{x_Nx_N}, \delta^{-2}(v_\delta)_{x_Nx_N}$-ones and, dropping the $\bar\eta>0$,
the limit problem for $\tilde u$ and $\tilde v$ is now
\begin{equation}\label{conv-conc}
-\tilde u_{x_Nx_N} \leq 0\leq -\tilde v_{x_Nx_N} \quad \hbox{if $ x_N>0$ and $ y_N>0$
respectively}\,.
\end{equation}
Notice that the strict subsolution property is lost in the limit here. On the boundary, we get
\[ \min(-\tilde u_{x_Nx_N}, -\tilde u_{x_N}+G(\xb',\tb, \pe ,\Ddxp   \tilde u)+\eta) \leq  0 ,\]
\[ \max(-\tilde v_{x_Nx_N}, -\tilde v_{x_N}+G(\yb',\sb, \pe ,\Ddxp   \tilde v)) \geq 0 ,\]
but using the uniform ellipticity in the normal direction of the equation inside the domain together with Proposition~\ref{VBC-uec}, these
relaxed boundary conditions reduce to
\[  -\tilde u_{x_N}+G(\xb',\tb, \pe ,\Ddxp   \tilde u)+\eta \leq  0 ,\]
\[  -\tilde v_{x_N}+G(\yb',\sb, \pe ,\Ddxp   \tilde v) \geq 0 .\]

Moreover, $\tilde u,\tilde v$ satisfy $\tilde u(0,0)=\tilde v(0,0)=0$ and since $\varphi(s)\leq s$,
\begin{equation}\label{ineq-quad}
\tilde u(x,t)-\tilde v(y,s)-\frac{|x'-y'|^2}{\e^2}-\frac{|t-s|^2}{\e^2}-\frac{L}{\e} |x_N-y_N|\leq 0.
\end{equation}

The difference in the second-order case is that the Lipschitz continuity of the subsolution in a
neighborhood of the boundary is not given for free and we are not sure that complete reduced sub-
and superjets do exist when tangential ones exist.

\medskip

\noindent\textbf{(d)} \textit{Regularization of $\tilde u,\tilde v$ and properties of the
regularized functions.}\\[2mm] { As in the first-order case, we perform a sup-convolution in
$(x',t)$ to $\tilde u$ and an inf-convolution in $(y', s)$ to $\tilde v$, keeping the same notations
(cf. \eqref{sup-conv}-\eqref{inf-conv}).
%
For $0<\alpha \ll \e$, we have the following\\ (i) $\tilde u^\alpha=\max(\bar
u^\alpha,(\psi^-)^\alpha)$, $\tilde v_\alpha=\min(\bar v_\alpha,(\psi^+)_\alpha)$.\\ (ii) Since
$(\psi^-)^\alpha=-K_1$ and $(\psi^+)_\alpha=+K_1$ if $x_N=0$, we have $\tilde u^\alpha (0,0)=\tilde
v_\alpha (0,0)=0$ and $\tilde u^\alpha =\bar u^\alpha $, $\tilde v_\alpha = \bar v_\alpha$ if
$x_N=0$ in a neighborhood of $(0,0)$.\\ (iii) $\tilde u^\alpha, \tilde v_\alpha$ satisfy the same
viscosity sub and supersolution inequalities as $\tilde u, \tilde v$ respectively: in fact, these
sup- and inf-convolution procedures do not present any technical difficulty since the nonlinearities
involved in the limiting problem do not depend neither on $x'$ nor on $t$.\\ (iv) Because $\tilde
u^\alpha, \tilde v_\alpha$ satisfy \eqref{conv-conc}, the functions $x_N\mapsto \tilde
u^\alpha((x',x_N),t)$ and $x_N\mapsto \tilde v_\alpha((x',x_N),t)$ are respectively convex and
concave for all $x',t$ in a neighborhood of $(0,0)$.

The next step consists in proving that  $\tilde u^\alpha,\tilde v_\alpha$ are continuous---and in
particular \wrt $x_N$---at any point $((x',0),t)$ in a neighborhood of $(0,0)$. We do it at $(0,0)$,
the proof being similar for the other points.

First, we apply Proposition~\ref{regul} to $\tilde u^\alpha, \tilde v_\alpha$, which yields
\begin{equation}\label{conseq-prop33} \tilde u^\alpha(0,0) =\limsup_{\substack{(x,t) \to (0,0) \\
    x_N>0}}\, \tilde u^\alpha(x,t), \qquad \tilde v_\alpha(0,0) =\liminf_{\substack{(x,t) \to (0,0)
    \\ x_N>0}}\, \tilde v_\alpha(x,t).  \end{equation} Then, we claim that $$ \tilde u^\alpha(0,0)
    =\lim_{\substack{(x,t) \to (0,0) \\ x_N>0}}\, \tilde u^\alpha(x,t), \qquad \tilde v_\alpha(0,0)
    =\lim_{\substack{(x,t) \to (0,0) \\ x_N>0}}\, \tilde v_\alpha (x,t).$$ Indeed, if $$ \tilde
    u^\alpha(0,0) >\underline l:= \liminf_{\substack{(x,t) \to (0,0)\\ x_N>0}}\, \tilde
    u^\alpha(x,t), $$ we can use a sequence $(x^k,t^k)_k$ converging to 0 such that $ \tilde
    u^\alpha(x^k,t^k) \to \underline l$.  Using the convexity of $\tilde u^\alpha$ implies that, if
    we fix some $\bar x_N>0$ and if we consider $x_N$ such that $\bar x_N>x_N>x_N^k$, we have $$
    \tilde u^\alpha(((x')^k,x_N),t^k) \leq \alpha^k \tilde u^\alpha(((x')^k,x_N^k),t^k)+
    (1-\alpha^k)\tilde u^\alpha(((x')^k,\bar x_N),t^k)\; ,$$ with $\alpha^k=(\bar x_N-x_N)(\bar
    x_N-x_N^k)^{-1}$. Letting $k \to +\infty$, we obtain \begin{equation}\label{wrongliminf} \tilde
        u^\alpha((0,x_N),0) \leq (\bar x_N-x_N)(\bar x_N)^{-1} \underline l + x_N(\bar
        x_N)^{-1}\tilde u^\alpha((0,\bar x_N),0)\; .  \end{equation} But $\tilde u^\alpha$ is
        uniformly Lipschitz continuous \wrt $x',t$ and, by \eqref{conseq-prop33}, we have $$\tilde
        u^\alpha(0,0)=\limsup_{\substack{x_N \to 0\\ x_N>0}}\tilde u^\alpha((0,x_N),0),$$ while
        \eqref{wrongliminf} yields $$\tilde u^\alpha(0,0)=\limsup_{\substack{x_N \to 0\\
        x_N>0}}\tilde u^\alpha((0,x_N),0) \leq \underline l < \tilde u^\alpha(0,0)\,,$$ which is a
        contradiction. Hence the claim is proved for $\tilde u^\alpha$ and a similar argument gives
        the same result for $\tilde v_\alpha$. And clearly all the above arguments are valid for any
        point of the form $(x',0,t))$ close enough to $(0,0)$.

Since $ \tilde u^\alpha,\tilde v_\alpha$ are both continuous at $(0,0)$, $\tilde u^\alpha=\bar
u^\alpha$, $\tilde v_\alpha=\bar v^\alpha$ in a neighborhood of $(0,0)$ since $(\psi^-)^\alpha=-K_1$
and $(\psi^+)_\alpha=+K_1$ if $x_N=0$ and these functions are bounded in a neighborhood of $(0,0)$.

Moreover \begin{equation}\label{keyineqbis2} \bar u^\alpha (x,t)-\bar v_\alpha (y,s) \leq \Cea
\left(\frac{|x'-y'|^2}{\e^2}+\frac{|t-s|^2}{\e^2}\right)+ \frac{L}{\e}|x_N- y_N|.  \end{equation}

Hence $(0,0,0,0)$ is still a maximum point of \[\bar u^\alpha (x,t)-\bar v_\alpha (y,s)-
\Cea\left(\frac{|x'-y'|^2}{\e^2}+\frac{|t-s|^2}{\e^2}\right)- \frac{L}{\e}|x_N- y_N|.\] As usual, we
may even assume that $(0,0,0,0)$ is a {\em strict} maximum point of this function by adding suitable
(small) terms.

 \medskip

\noindent\textbf{(e)} \textit{Adapting the Ishii-Jensen Lemma.}\\[2mm] Now, for $q\in \R^{2N}$ close
to $0$, we consider the functions 
 \[\bar u^\alpha (x,t)-\bar v_\alpha (y,s)-
\Cea\left(\frac{|x'-y'|^2}{\e^2}+\frac{|t-s|^2}{\e^2}\right)-L\varphi\left(\frac{|x_N-y_N|}{\e}\right)-q\cdot
(x',t,y',s).\]         Arguing as in the first step---this is even easier here---, all these functions
achieve their maximum at points such that $x_N=y_N=0$. Then, by combining Theorem~A.2 and Lemma~A.5
in \cite{Users} in the tangent variables, there exists a sequence $(q_k)_k$ of points in $\R^{2N}$
such that each function 
 \[\bar u^\alpha (x,t)-\bar v_\alpha (y,s)-
\Cea\left(\frac{|x'-y'|^2}{\e^2}+\frac{|t-s|^2}{\e^2}\right)-
L\varphi\left(\frac{|x_N-y_N|}{\e}\right)-q_k\cdot (x',t,y',s)\] 
has a maximum point at
$((x'_k,0),t_k,(y'_k,0),s_k)$ where $\bar u^\alpha,\bar v_\alpha$ are twice differentiable \wrt the
$x'$-variable.

At these points, we have full reduced super and subjets for $\bar u^\alpha,\bar v_\alpha$
respectively (we recall that these were defined in Section~\ref{sec:jets}), because of this maximum
point property.

We denote by $X_k'=\Ddxp   \bar u^\alpha ((x'_k,0),t_k)$ and $Y_k'= \Ddxp   \bar v_\alpha ((y'_k,0),s_k)$; they satisfy
\begin{equation}\label{ineqmat3}
\left[\begin{array}{cc}X_k'&0\\0&-Y_k'\end{array}\right]
 \le \frac 2 {\e^2} \Cea \left[\begin{array}{cc}I&-I\\-I&I\end{array}\right].
\end{equation}
Taking into account the fact that the boundary condition just depends on the first
derivative \wrt $x_N$ and the second derivative in $x'$, as we already mentioned in
Section~\ref{sec:jets} we reduce the semijets by dropping the $p_x$ and $p_t$-terms, keeping only
the $p_{x_N}$ and $M$-ones. 

Following Proposition~\ref{prop:jets} on the structure of semijets, we introduce
the sets $\Lambda^+ (u):=\{\lambda_1: (\lambda_1,X_k')\in \PP^{2,+}_r \bar u^\alpha
((x'_k,0),t_k)\}$ and $\Lambda^-(v):=\{\lambda_2:(\lambda_2,Y_k')\in \PP^{2,-}_r \bar v_\alpha
((y'_k,0),s_k)\}$. This proposition proves that if $\overline \lambda^k_1:=\inf  \Lambda^+ (u)$,
$\underline\lambda^k_2:=\sup \Lambda^- (v)$, then
\begin{enumerate}
    \item[$(i)$] $(\lambda_1,X_k')\in \PP^{2,+}_r \bar u^\alpha  ((x'_k,0),t_k)$ if $\lambda_1 > \overline \lambda^k_1$,
    \item[$(ii)$] $(\lambda_2,Y_k')\in \PP^{2,-}_r \bar v_\alpha ((y'_k,0),s_k)$ if $\lambda_2 > \underline\lambda^k_2$.
\end{enumerate}

Because of the uniform ellipticity in $x_N$ of the equation in the domain, the boundary condition is
satisfied in a strong sense---see Proposition~\ref{VBC-uec}---and we have
\begin{equation}\label{ineq.bdry.approx}
    \begin{cases} -\lambda_1 +G\big(x'_k,t_k, \pe ,X_k'\big)\leq  -\eta
        \quad \hbox{for any   } \lambda_1 \geq \overline \lambda^k_1\,,\\
 -\lambda_2+G\big(y'_k,s_k, \pe ,Y_k'\big) \geq 0 \quad \hbox{for any   }
        \lambda_2 \leq \underline \lambda^k_2\,,
    \end{cases}\end{equation}
And therefore $\overline \lambda^k_1, \underline\lambda^k_2$ are necessarily finite.

Now we use several estimates: on one hand, we know that
$$
\frac{|\bar x'-\bar y'|^2}{\e^2}+\frac{|\bar t-\bar s|^2}{\e^2}\to 0 \quad \hbox{as  }\e \to 0.
$$
and, on the other hand,
$$(x'_k,t_k)-(\bar x',\tb), (y'_k,s_k)-(\bar y',\sb)=o_\alpha(1)+o_k(1)\; .$$
Using \hyp{Cont}, this yields
$$G\big(x'_k,t_k, \pe ,X_k'\big) - G\big(y'_k,s_k, \pe ,Y_k'\big) \geq o_\e(1) + o_\alpha^{(\e)}(1) +o_k^{(\e)}(1)\; ,$$
where $o_\alpha^{(\e)}(1) +o_k^{(\e)}(1) \to 0$ if $\alpha \to 0, k\to \infty$ with a fixed $\e$ and $o_\e(1)\to 0$ when $\e \to 0$.
The above inequalities lead to
\[ o_\e(1) + o_\alpha^{(\e)}(1) +o_k^{(\e)}(1) \leq \overline \lambda^k_1 - \underline
    \lambda^k_2-\eta .\]
At this point, we want to make precise our use of the parameters $\e, \alpha$ and $k$: we first choose $\e$ in order to have
the above $o_\e(1)$ to be less that, say, $\eta/4$, then we have to choose $\alpha$ small enough and $k$ large enough
compared to $\e$.

With this choice of the parameters, we can assume without loss of generality
that $\overline \lambda^k_1 - \underline \lambda^k_2 \geq \eta/2$.}

\medskip

\noindent\textbf{(f)} \textit{Getting a contradiction.}\\[2mm]
We first recall that $\overline \lambda^k_1$ is the infimum of the $\lambda_1$ such that $(\lambda_1, X_k')$ is in the
reduced superjet of $\bar u^\alpha$ while $\underline \lambda^k_2$ is the supremum of $\lambda_2$ such
    that $(\lambda_2, Y_k')$ is in the reduced subjet of $\bar v_\alpha$. We use this information below in a crucial way.

Then, we notice that the functions $x_N\mapsto \bar u^\alpha(x',x_N,t)$ are convex for any $x',t$ close to $(0,0)$ and, in the
same way, the functions $x_N\mapsto \bar v_\alpha (x',x_N,t)$ are concave for any $x',t$ close to $(0,0)$---hence close to $((x'_k,0),t_k)$ and $((y'_k,0),s_k)$ respectively. Therefore, since these functions are bounded, they are locally Lipschitz continuous and their derivatives (defined almost everywhere) are non-decreasing and non-increasing respectively.

    Moreover, as we have seen it above, the functions $x_N\mapsto \bar u^\alpha(x',x_N,t)$ and $x_N\mapsto \bar v_\alpha(x',x_N,t)$ are necessarily continuous at $x_N=0$. 
    

Now we claim that there does not exist a neighborhood $\mathcal{V}$ of $((x'_k,0),t_k)$ such that
$$
    \frac{\partial \bar u^\alpha((x',x_N),t)}{\partial x_N} \leq \overline \lambda^k_1-\eta/8\quad \hbox{for all $(x',t)$, a.e. in $x_N$} ,
$$
if $((x',x_N),t)\in \mathcal{V}\cap (\Omega \times \R)$.

    Indeed otherwise we would have a contradiction with the definition of $\overline \lambda^k_1$. 
    In the same way, the property
$$
\frac{\partial \bar v_\alpha
    ((x',x_N),t)}{\partial x_N} \geq \overline \lambda^k_2+\eta/8\quad \hbox{for all $(x',t)$, a.e. in $x_N$} ,
$$
cannot hold for $((x',x_N),t)\in \mathcal{V}\cap (\Omega \times \R)$, where $\mathcal{V}$ a neighborhood of $((y'_k,0),s_k)$.

    Hence, there exists a sequence $(x^{(p)},t^{(p)})_p$ converging to $((x'_k,0),t_k)$ such that
 $$  \frac{\partial \bar u^\alpha(x^{(p)},t^{(p)})}{\partial x_N} \geq \overline \lambda^k_1-\eta/8 $$
 and using the convexity of $x_N\mapsto \bar u^\alpha(x',x_N,t)$, we have, for all $x_N\geq x^{(p)}_N$
 $$  \bar u^\alpha(((x^{(p)})',x_N),t^{(p)})\geq \bar u^\alpha(x^{(p)},t^{(p)})+ (\overline \lambda^k_1-\eta/8 )(x_N-(x^{(p)})_N)\; .$$
   And, of course, we have a similar inequality for $\bar v_\alpha$ by using its concavity.
   
 We can pass to the limit in these inequalities by using the continuity of  $\bar u^\alpha, \bar v_\alpha$ at $((x'_k,0),t_k)$ and the
 tangential continuity of both functions; we finally obtain
 \[\begin{cases}
        \bar u^\alpha((x'_k,x_N),t_k)-\bar u^\alpha((x'_k,0),t_k) \geq
        (\overline\lambda^k_1-\eta/8)x_N  ,\\
        \bar v_\alpha ((y'_k,x_N),s_k)-\bar v_\alpha ((y'_k,0),s_k)\leq
        (\overline\lambda^k_2+\eta/8)x_N,
    \end{cases}\]
    leading to 
    \[\begin{aligned} 
        \Big[\bar u^\alpha((x'_k,x_N),t_k)-\bar v_\alpha ((y'_k,x_N),s_k)\Big] & - 
        \Big[\bar u^\alpha((x'_k,0),t_k)-\bar v_\alpha ((y'_k,0),s_k)\Big] \\
        & \geq \Big(\overline \lambda^k_1-\overline \lambda^k_2 -\frac{\eta}{4}\Big)x_N \geq \frac{\eta}{4}\,x_N  .
    \end{aligned}
    \]
Letting $k\to \infty$, this yields, by using the tangential continuity of $\bar u^\alpha,\bar v_\alpha$
    \[\bar u^\alpha(0,x_N,0)-\bar v_\alpha(0,x_N,0) \geq \frac{\eta}{4}\, x_N,\]
but \eqref{keyineq} implies $\bar u^\alpha(0,x_N,0)-\bar v_\alpha(0,x_N,0) \leq 0$, and we reach a
contradiction which ends the proof.

\section{Further Results and Open Questions}\label{FROR}

In this last section we gather some comments, open questions and other results
concerning Problem~\eqref{Eqn}-\eqref{Vbc-GC}. 

\subsection*{Existence via Perron's method}

We first provide an {\em existence result} for Problem~\eqref{Eqn}-\eqref{Vbc-GC} associated to the
initial condition \eqref{init} and to do so we use the assumption

\begin{app}{\hyp{B-Ex}}{Boundedness assumption for existence.}{\smsp
    The functions $x\mapsto u_0(x)$, $(x,t)\mapsto F(x,t,0,0)$ and $(x,t)\mapsto \G(x,t,0,0)$ are
    bounded and continuous on $\Omegb$, $\Omegb \times [0,T]$ and $\domeg \times [0,T]$ respectively.}
\end{app}

The result is the
\begin{proposition}\label{exist} 
    Under the assumptions of Theorem~\ref{comp:GC}, if \hyp{B-Ex} holds, there exists a unique,
    bounded continuous solution to Problem~\eqref{Eqn}-\eqref{Vbc-GC}-\eqref{init}.
\end{proposition}

\begin{proof}
    We just give the main arguments since the proof is based on the classical Perron's method (\cf
    Ishii~\cite{Is-Per}, see also \cite{Users}).

    The key point is to build sub and supersolutions of the problem and they have the form
    \[ u^\pm (x,t):=\pm k_1 t \pm k_2 \varphi(d(x))+ k_3 ,\]
    where $\varphi$ is $\varphi_1$ defined at the beginning of Section~\ref{sec:GCR.LCR} and $k_1,k_2,k_3$
    are constant which are chosen in the following way: 
    \begin{enumerate}
    \item[$(i)$] $k_2$ is chosen in order to have $u^\pm$ satisfying the Ventcell boundary
        condition, \cf Lemma~\ref{dist-func}.
    \item[$(ii)$] Then $k_1$ is chosen in order to ensure that $u^\pm$ are sub and supersolution of \eqref{Eqn}.
    \item[$(iii)$] Finally $k_3$ is chosen in order to have
        \[u^-(x,0)\leq u_0(x)\leq u^+(x) \quad \hbox{on  }\Omegb.\]
    \end{enumerate}
    With all these properties, one can apply Perron's method---with an initial data being understood
    in the viscosity sense. And the result is proved.
\end{proof}

\subsection*{Including some $u_t$-dependence in $G$}

It is clear that boundary conditions like
\begin{equation}\label{Vutgen}
\G(x,t,u_t,Du,D_T^2u)= 0 \quad \hbox{on }\domeg \times (0,T) ,
\end{equation}
where $\G(x,t,p_t,p,M_T)$ is an increasing function in $p_t$ can be treated analogously, typically
\begin{equation}\label{Vutpart}
u_t-\frac{\partial u}{\partial x_N} + G(x',t,\Dxp  u,\Ddxp  u)= 0 \quad \hbox{on }\{x_N=0\} \times (0,T).
\end{equation}
The assumptions on the dependence in $p_t$ are analogous to those made on the tangential part of $p$ since, as it is 
already the case in this article, $t$ can be seen as a tangent variable to the boundary $\domeg \times (0,T)$.

We refer the reader to Remark~\ref{simple-pb} in order to be convinced that the $u_t$-dependence does no create any difficulty
in the case of \eqref{Vutpart} neither in the first-order equation case, nor in the second-order one. On the contrary, for \eqref{Vutgen}, we have
to impose quasiconvexity assumptions on $\G(x,t,p_t,p,M_T)$ \wrt $(p_t, p, M_T)$, hence we have, in particular, quite an unusual assumption on the
dependence \wrt $u_t$.

\subsection*{The stationary case}

We point out that the stationary case can be treated analogously provided that the nonlinearity of
the equation is proper in the sense of \cite{Users}. We are not going to give any detail here but
both the existence and comparison result hold in this framework, as the reader will certainly be
able to check.

\subsection*{A few open questions}

Via \hyp{Cont}, we assume the same regularity for $F(x,t,p,X),G(x',t,p',X')$ in $x$ or $x'$ and $t$.
We have no idea if this assumption is really necessary or if one can replace it by some weaker
continuity requirement for the $t$-variable.

In the same way, the Lipschitz continuity assumption in \hyp{Gen}-$(i)$ may be seen as natural for
$G$ or $\G$ as part of the requirement for a ``good Ventcell boundary condition'', the linear growth
in $p$ and $M_T$ ensuring---in some sense---that the normal derivative can control them. However,
this assumption seems less natural for $F$ which, for example, may have some superlinear gradient
growth, which is incompatible with \hyp{Gen}-$(i)$. We do not address this question here but it is clearly
a problem to be considered.

The $C^{0,\alpha}$-regularity of solutions for $\alpha \in (0,1]$ is an interesting question which
is also a prerequisite to address other problems like the large time behavior of solutions via the
study of the {\em ergodic problem}. 

Considering the methods we used to get the comparison result suggests that these regularity results
should follow from similar ideas.

\section*{Appendix}

\subsection{Computation of sup-convolutions}
\label{app:supconv}

In various places in the paper we use the following computation (on the tangential variables
$x'$ and $t$)
\begin{lemma}
    For any $\e>\alpha>0$, we have
    \[
        \sup_{z\in\R^{N}}\Big(\frac{|z-y|^2}{\e^2}-\frac{|x-z|^2}{\alpha^2}\Big)=\Cea\,\frac{|x-y|^2}{\e^2},
    \]
    where $\Cea:=(1-\frac{\alpha^2}{\e^2})^{-1}$.
\end{lemma}
\begin{proof} We first remark that, since $\e>\alpha>0$, the function 
$$ z \mapsto \frac{|z-y|^2}{\e^2}-\frac{|x-z|^2}{\alpha^2},$$
is coercive and therefore the sup is actually attained at some point $z_0\in \R^N$. Then
the proof is just a simple calculation: indeed, $z_0$ satisfies the equation
    $\frac{z_0-y}{\e^2}=\frac{z_0-x}{\alpha^2}$, which also yields
    $(1-\frac{\alpha^2}{\e^2})(z_0-y)=(x-y)$. The result directly follows.
\end{proof}


\subsection{Convex combination of subsolutions}
\label{app:convcomb}

We are interested here in Lipschitz continuous subsolutions of the following Ventcell problem for some $r>0$
\begin{equation}\label{eqn-ins}
\tilde F(u_{x_N}) \leq 0 \quad \hbox{in }\Omega_r:=\{x=(x',x_N): x'\in B(0,r), x_N\in (0,r)\},
\end{equation}
\vspace*{-4mm}
\begin{equation}\label{eqn-bound}
-u_{x_N}+\tilde G(\Dxp  u, \Ddxp  u) \leq 0 \quad \hbox{in }\H_r:=B(0,r)\times \{0\},
\end{equation}
where $\tilde F(p_N), \tilde G(p',X')$ are functions satisfying some assumptions below.
To state and prove the results of this appendix, we have dropped the $t$-variable for the sake of simplicity. But it is clear enough that
this variable plays no role here.

\begin{theorem}\label{comb-conv} 
    Assume that $p_N\mapsto \tilde F(p_N)$ and $(p',M')\mapsto\tilde
    G(p',M')$ are continuous, quasiconvex functions and that $\tilde G$ satisfies the ellipticity
    condition. If $u_1, u_2$ are Lipschitz continuous subsolutions of
    \eqref{eqn-ins}-\eqref{eqn-bound}, then, any convex combination $w:=\rho_1 u_1 +\rho_2 u_2$
    ($\rho_1,\rho_2\geq 0$, $\rho_1+\rho_2=1$) is also a subsolution of
    \eqref{eqn-ins}-\eqref{eqn-bound}.
\end{theorem}

Clearly the main interest of this result is that the convex combination $w$ is a subsolution {\em up to the boundary}; inside the domain,
the result is classical and the first versions already appear in the book of Lions \cite{L} for convex Hamiltonian, mainly using the fact that a
$W^{1,\infty}$-function which satisfies the subsolution inequality in the almost everywhere sense is a viscosity subsolution. And the extension
to quasiconvex nonlinearities does not present any additional difficulty since it follows from the same argument.

But, of course, in more general contexts like second-order equation or, as it is the case here, with a boundary condition in the viscosity sense,
the situation is completely different and one faces a non-trivial difficulty since we have to argue with a doubling of variables which both requires
a good control of the penalization terms and an ad hoc treatment of the boundary condition. And actually, even if we consider a Neumann boundary 
condition instead of a Ventcell one, we do not know how to obtain a far more general result than Theorem~\ref{comb-conv}.

Last (but not least) remark: for us, the main interest of Theorem~\ref{comb-conv} is to be the main step in order to the tangential regularization
of subsolutions (cf. Corollary~\ref{regul-smooth} below).

\begin{proof} Let $\phi$ be a smooth test-function and $\xb=(\xb',\xb_N)\in \Omega_r\cup \H_r$ be a strict local maximum point of
    $w-\phi$ on $\overline{\Omega}_r$. Of course, as we mention it above, the only difficulty is when $\xb\in \H_r$ since, if $\xb\in \Omega_r$, the result is classical.
Hence we may assume that $\xb_N=0$.

Our aim is to prove that
$$ \min \left( \tilde F(\phi_{x_N} (\xb)), -\phi_{x_N}(\xb)+\tilde G(\Dxp \phi(\xb),\Ddxp  \phi(\xb) \right) \leq 0.$$
To do so, we assume that $\tilde F(\phi_{x_N} (\xb))>0$ and we are going to show that
\begin{equation}\label{expect-ineq}
-\phi_{x_N}(\xb)+\tilde G(\Dxp \phi(\xb),\Ddxp  \phi)(\xb) \leq 0.
\end{equation}

We first use Lions and Souganidis arguments (cf. \cite{LiSo1,LiSo2}, see also \cite{BCbook}): since $\tilde F(w_{x_N}) \leq 0$
in $\Omega_r$, if
$$ \underline p=\liminf \left[\frac{w(\xb',x_N)-w(\xb',0)}{x_N}\right] \quad\hbox{and}\quad \overline p=\limsup \left[\frac{w(\xb',x_N)-w(\xb',0)}{x_N}\right],$$
we have $\tilde F(p) \leq 0$ for any $\underline p\leq p \leq \overline p$. Moreover, the maximum point property of $\xb$ also
implies that $\overline p \leq \phi_{x_N}(\xb)$ and, by the quasiconvexity of $\tilde F$, we deduce that $\tilde F$ is nondecreasing in a
neighborhood of $[\phi_{x_N}(\xb),+\infty)$.

Then, in order to show that \eqref{expect-ineq} holds, we triple the variables and consider the function $\Phi  (x_1,x_2,x)$ given by
$$\rho_1 (u_1(x_1)-\phi (x_1)) +\rho_2 (u_2 (x_2)-\phi (x_2))- \frac{|x_N-x_N^1|^2}{\e}-\frac{|x_N-x_N^2|^2}{\e}$$
$$ - \frac{|x'-x_1'|^2}{\beta}-\frac{|x'-x_2'|^2}{\beta},$$
with $x=(x',x_N), x_1=(x_1',x_N^1), x_2=(x_2',x_N^2)$ and where $\e,\beta>0$ are devoted to tend to $0$.

We look at maximum points of this function \wrt all the variables. Because of the strict maximum point property of $\xb$, there exists a sequence of maximum points
of this function which converges to $(\xb,\xb,\xb)$ when $\e, \beta \to 0$; to simplify the notations, we drop the dependence in $\e$ and $\beta$
and we just denote such a maximum point by $(x,x_1,x_2)$.

We have two cases.

$\bullet$ If $x_N>0$, then the maximum property in $x_N$ yields
$$ x_N=\frac12( x_N^1+x_N^2).$$
If $x_N^1>x_N$, then $x_N^1>0$ and, since $u_1$ is a subsolution for $\tilde F$, we have 
$$ \tilde F\left(\phi_{x_N} (x_1) + \frac{2(x^1_N-x_N)}{\rho_1 \e} \right)\leq 0.$$
But this inequality cannot holds for $\e,\beta$ small enough since $\phi_{x_N} (x_1)$ is close to $\phi_{x_N} (\xb)$,
 $\tilde F(\phi_{x_N} (\xb))>0$ and $\tilde F$ is nondecreasing in a
neighborhood of $[\phi_{x_N}(\xb),+\infty)$.

Therefore this case cannot happen and necessarily $x_N^1\leq x_N$. The same argument also shows that
$x_N^2\leq x_N$ so that $x_N^1=x_N^2=x_N$. But, for a similar reason as above, none of the
inequalities $ \tilde F(\phi_{x_N} (x_1)), \tilde F(\phi_{x_N} (x_2))\leq 0$ can be true so
we conclude that this case cannot hold, and we are necessarily in the case $x_N=0$.

 $\bullet$ If $x_N=0$, the maximum property in $x_N$ implies the (apparently) weaker inequality
$$ x_N=0 \geq \frac12( x_N^1+x_N^2).$$
But, of course, this immediately implies that $x_N^1=x_N^2=x_N=0$.

In order to obtain the right viscosity subsolution inequalities, we notice that, on one hand, we can assume without loss of generality that
$u_1,u_2$ are semi-convex in the $x'$-variable and, on the other, that the above arguments apply as well to any function of the type
$\Phi(x_1,x_2,x) -q_1\cdot x_1'-q_2\cdot x_2'$ where $q_1,q_2 \in \R^{N-1}$ are close to $0$. With these two remarks, the Jensen-Ishii Lemma
can be applied without any difficulty (cf. \cite{Users}) and yield the existence of $p'_1,p'_2\in \R^{N-1}$ and
$(N-1)\times(N-1)$ symmetric matrices $X_1,X_2$ such that
$$\min( \tilde F(\phi_{x_N} (x_1)),-\phi_{x_N}(x_1)+\tilde G(p'_1,X_1)\leq 0,$$
$$\min( \tilde F(\phi_{x_N} (x_2)),-\phi_{x_N}(x_2)+\tilde G(p'_2,X_2)\leq 0,$$
where, for $i=1,2$
$$p'_i = \Dxp \phi(x_i)+ \frac{2(x'_i-x')}{\rho_i\beta},$$
and where $X_1,X_2$ satisfy, for any $r_1,r_2,r \in \R^{N-1}$
$$\rho_1 \left( X_1-\Ddxp  \phi (x_1)\right)r_1\cdot r_1 +\rho_2 \left(X_2-\Ddxp   \phi (x_2)\right)r_2\cdot r_2\leq
 \frac{|r-r_1|^2}{\beta}-\frac{|r-r_2|^2}{\beta}.$$
By choosing $r=r_1=r_2$, we deduce that
$$ \rho_1 X_1 +\rho_2 X_2 \leq \rho_1 \Ddxp   \phi (x_1) +\rho_2 \Ddxp   \phi (x_2) =
\Ddxp   \phi (\xb)+ o_{\e,\beta}(1)\; ,$$
and we also have
\begin{align*} 
\rho_1 p'_1 +\rho_2 p'_2= & \rho_1 \Dxp \phi(x_1)+\rho_2 \Dxp \phi(x_2)+\frac{2(x'_1-x')}{\beta}+\frac{2(x'_2-x')}{\beta}\\
    =& \rho_1 \Dxp \phi(x_1)+\rho_2 \Dxp \phi(x_2)=\Dxp \phi(\xb)+o_{\eps,\beta}(1),
\end{align*}
since, by the maximum property for $x'$, we have $2x'=x'_1+x'_2$.

In the above subsolutions inequalities, we notice that, for $\e,\beta$ small enough, we have
$$ \tilde F(\phi_{x_N} (x',x_N^1)),
\tilde F(\phi_{x_N} (y',x_N^2))> 0,$$ and therefore
$$-\phi_{x_N}(x_1)+\tilde G(p'_1,X_1)\leq 0,$$
$$-\phi_{x_N}(x_2)+\tilde G(p'_2,X_2))\leq 0.$$
The conclusion follows easily since, by quasiconvexity, 
$$\tilde G (\rho_1 p'_1 +\rho_2 p'_2,\rho_1 X_1 +\rho_2 X_2) \leq \max( \tilde G(p'_1,X_1),\tilde G(p'_2,X_2)),$$ and that
$\phi_{x_N}(x_1), \phi_{x_N}(x_2)\to \phi_{x_N}(\xb)$ when $\e,\beta \to 0$.
\end{proof}

\subsection{Regularization of subsolutions}
\label{app:regsubsol}

As a quasi-immediate corollary of Theorem~\ref{comb-conv}, we have the following regularization result in which $(\rho_\eps)_\eps$ denotes
a sequence  of positive, $C^\infty$-functions on $\R^{N-1}$, $\rho_\eps$ having a compact support in $B(0,\eps)$ and with
$\int_{\R^{N-1}}\,\rho_\eps(e)de=1$. 
\begin{corollary}\label{regul-smooth} Under the assumptions of Theorem~\ref{comb-conv} on $\tilde
    F,\ \tilde G$, let $u$ be a Lipschitz continuous subsolution of \eqref{eqn-ins}-\eqref{eqn-bound}. 
    If, for $\e\ll r$, $u^\e: B(0,r-\e)\times (0,r)\to \R$ is given by
    $$u^\e(x):= \int_{|e| <\eps} u(x'-e,x_N)\rho_\eps(e)de\,,$$
    then $u^\e$ is a Lipschitz continuous subsolution of \eqref{eqn-ins}-\eqref{eqn-bound} in 
    $B(0,r-\e)\times (0,r)$. Moreover, for any $0\leq x_N <r$, $x'\mapsto u^\e(x',x_N)$ is a
    smooth function and $\Dxp \ue(x),  \Ddxp  \ue(x)$ are continuous functions of $x$.
\end{corollary}

\begin{proof}
    Let us begin by mentioning that of course, by induction Theorem~\ref{comb-conv} can be generalized to any
    (finite) convex combination of subsolutions $(u_k)$. 

    Considering now a discretization of the
    convolution, we see that $u^\e$ can be approximated by a finite sum: for any $\eta>0$ there
    exists $N\in\N$, $(\mu_k)_{k=1..N}$ and $(e_k)_{k=1..N}$ such that for any $k=1..N$, $\mu_k\geq0$,
    $\sum_{k=1}^N\mu_k=1$, and the function
    $$u^\e_N(x):=\sum_{k=1}^N \mu_k u(x'-e_k,x_N)$$ satisfies
    $|u_N^\e-u^\e|\leq \eta$ locally uniformly in 
    $B(0,r-\e)\times (0,r)$.

    Notice that since $\tilde F$ and $\tilde G$ do not depend on $x$, for any $k=1..N$,
    $x\mapsto u(x'-e_k,x_N)$ is a subsolution of \eqref{eqn-ins}-\eqref{eqn-bound}, so that 
    for any $N\in\N$, $u_N^\e$ is also a subsolution of \eqref{eqn-ins}-\eqref{eqn-bound}. Finally,
    since $u_N^\e\to u^\e$ locally uniformly, we use the stability property of viscosity solutions
    to conclude that $u^\e$ is also a viscosity subsolution of \eqref{eqn-ins}-\eqref{eqn-bound} in
    $B(0,r-\e)\times (0,r)$. Of course, the regularity property holds as a result of the
    convolution, the function $u$ being continuous itself.
\end{proof}

\end{document}